\documentclass[12pt]{article}

\usepackage{amsmath}
\usepackage{amssymb}
\usepackage{amscd}
\usepackage{amsthm}
\usepackage{indentfirst}
\usepackage[latin1]{inputenc}
\usepackage[english,frenchb]{babel}
\usepackage{pst-node}

\theoremstyle{plain}
\newtheorem{thm}{Th\'eor\`eme}[section]
\newtheorem{lem}[thm]{Lemme}
\newtheorem{cor}[thm]{Corollaire}
\newtheorem{prop}[thm]{Proposition}

\theoremstyle{definition}
\newtheorem{dfn}[thm]{D\'efinition}
\newtheorem{ntt}[thm]{Notations}

\newtheorem{rmq}[thm]{Remarque}
\newtheorem{exm}[thm]{Exemple} 
\newtheorem{sublem}[thm]{Sous-lemme}
\usepackage{mathrsfs}

\DeclareMathOperator{\GL}{GL}

\DeclareMathOperator{\Gal}{Gal}

\DeclareMathOperator{\End}{End}

\DeclareMathOperator{\card}{card}

\DeclareMathOperator{\Jac}{Jac}

\newcommand{\ie}{{\em i. e.\ }}

\newcommand{\F}{\mathbb{F}}
\newcommand{\Z}{\mathbb{Z}}
\newcommand{\Q}{\mathbb{Q}}

\newcommand{\Pro}{\mathbb{P}}

\newcommand{\coupl}[1]{\langle #1 \rangle}
\newcommand{\T}{\mathbb{T}}

\newcommand\topbot[2]{{\genfrac{}{}{0pt}{}{{#1}}{{#2}}}}
\newcommand{\legendre}[2]{
\genfrac{(}{)}{0.5pt}{}{#1}{#2}}

\begin{document}

\title{Points rationnels sur les quotients d'Atkin-Lehner de courbes de Shimura de discriminant $pq$}

\author{Florence Gillibert}

\maketitle

\selectlanguage{english}
\begin{abstract} 
Let $p$ and $q$ be two distinct prime numbers, and $X^{pq}/w_q$ be the quotient of the Shimura curve of discriminant $pq$ by the Atkin-Lehner involution $w_q$. We describe a way to verify in wide generality a criterion of Parent and Yafaev to prove that if $p$ and $q$ satisfy some explicite congruence conditions, known as the conditions of the non ramified case of Ogg, and if $p$ is large enough compared to $q$, then the quotient $X^{pq}/w_q$ has no rational point, except possibly special points. 
\end{abstract}

\selectlanguage{francais}
\begin{abstract} 
Soient $p$ et $q$ deux nombres premiers distincts et $X^{pq}/w_q$ le quotient de la courbe de Shimura de discriminant $pq$ par l'involution d'Atkin-Lehner $w_q$. Nous d\'ecrivons un moyen permettant de v\'erifier un crit\`ere de Parent et Yafaev en grande g\'en\'eralit\'e pour prouver que si $p$ et $q$ satisfont des conditions de congruence explicites, connues comme les conditions du cas non ramifi\'e de Ogg, et si $p$ est assez grand par rapport \`a $q$, alors le quotient $X^{pq}/w_q$ n'a pas de point rationnel non sp\'ecial. 
\end{abstract}

\section{Introduction}
La recherche de quotients d'Atkin-Lehner de courbes de Shimura sans point rationnel non sp\'ecial a \'et\'e l'objet de plusieurs travaux notamment de Clark \cite{Clarkthesis}, Rotger \cite{Rotger}, Rotger, Skorobogatof et Yafaev \cite{RSY}, Bruin, Flynn, Gonz\'alez et Rotger \cite{BFGR}, ainsi que de Parent et Yafaev \cite{ParentYafaev}. Pour des \'equations de ces courbes, on peut aussi consulter les travaux (\`a para\^itre) de Molina \cite{Molina}. Ce probl\`eme est li\'e \`a une conjecture (attribu\'ee \`a Coleman) sur les anneaux d'endomorphismes potentiels de vari\'et\'es ab\'eliennes de type $\GL_2$ (cf. \cite{BFGR}). On dit qu'une vari\'et\'e ab\'elienne $A/\Q$ est de type $\GL_2$ si son alg\`ebre d'endomorphisme $\End_{\Q}(A)\otimes \Q$ est un corps de nombre de degr\'e $\dim(A)$. La conjecture est \'enonc\'ee ainsi par Clark et Mazur : pour toute dimension $g$ fix\'ee, il y a un nombre fini de classes d'isomorphismes d'anneaux d'endomorphismes sur $\overline{\Q}$ de vari\'et\'es ab\'eliennes $A/\Q$ de type $\GL_2$ de dimension $g$. Dans le cas $g=1$, ces vari\'et\'es ab\'eliennes sont les courbes elliptiques, et la liste finie des anneaux d'endomorphismes possibles est connue. L'\'etude des courbes de Shimura permet d'aborder cette conjecture dans le cas de la dimension $g=2$. En effet, \'etant donn\'es une alg\`ebre de quaternion $B_D$ de discriminant $D$, et un entier sans facteur carr\'e $m$, toute surface ab\'elienne $A/\Q$, telle que $\End_{\Q}(A)\otimes \Q=\Q(\sqrt{m})$ et $\End_{\overline{\Q}}(A)$ soit un ordre maximal de $B_D$, est associ\'ee \`a un point rationnel non sp\'ecial sur le quotient $X^D/w_m$ de la courbe de Shimura de discriminant $D$ par l'involution d'Atkin-Lehner $w_m$ (cf. \cite[thm 4.5]{BFGR}). 
L'absence de point rationnel non sp\'ecial sur $X^D/w_m$ est en fait un probl\`eme plus fort que l'absence de telles surfaces ab\'eliennes $A/\Q$; en effet de tels points rationnels peuvent correspondre \`a des surfaces ab\'eliennes \`a multiplication quaternionique d\'efinies sur une extension de $\Q$ et admettant $\Q$ comme corps de module. 

Dans cet article, nous nous restreignons au cas o\`u le discriminant $D$ de la courbe de Shimura est le produit de deux nombres premiers $p$ et $q$. L'existence de points rationnels sur $X^{pq}/w_q$ peut se produire dans deux cas d\'ecrits par Ogg (cf. proposition \ref{casdeogg}) : le \og cas ramifi\'e\fg \ et le \og cas non ramifi\'e\fg . Nous nous pla\c cons dans le cas non ramifi\'e de Ogg. Remarquons que dans ce cas l'absence de surfaces ab\'eliennes \`a multiplication quaternionique d\'ecoule d'un r\'esultat de Rotger : \cite[thm 1.4]{Rotger}. 
Dans leur article \cite{ParentYafaev}, Parent et Yafaev donnent un crit\`ere, pour l'absence de point rationnel non sp\'ecial sur un quotient d'Atkin-Lehner $X^{pq}/w_q$ de courbe de Shimura dont le discriminant est le produit de deux nombres premiers. Nous rappelons leur \'enonc\'e dans la proposition \ref{parentyafaev}. Une application imm\'ediate de ce crit\`ere leur permet de d\'eterminer un exemple de famille infinie de quotients d'Atkin-Lehner de courbes de Shimura sans point sp\'ecial, de la forme $X^{251p}/w_{251}$, mais la v\'erification de ce crit\`ere n'\'etait jusqu'alors possible que dans de tels cas particuliers. Dans cet article nous appliquons ce crit\`ere \`a des courbes $X^{pq}/w_q$ beaucoup plus g\'en\'erales. En fait pour presques toutes \`a $q$ fix\'e : pour tout nombre premier $q>245$, et tout nombre premier $p$ sup\'erieur \`a une borne (non effective) d\'ependant de $q$, si la condition du cas non ramifi\'e de Ogg est satisfaite (cf. proposition \ref{casdeogg} ci-dessous), alors nous pouvons appliquer le crit\`ere. 

\begin{thm}
\label{resfinal}
Soit $q>245$ un nombre premier avec $q \equiv 3 \mod 4$. Il existe une borne $B_q$ d\'ependant de $q$ telle que si $p$ est un nombre premier v\'erifiant $p\equiv 1 \mod 4$, $\legendre{p}{q}=-1$ et $p\ge B_q$, alors la courbe $X^{pq}/w_q$ est sans point rationnel non sp\'ecial. 
\end{thm}

R\'esumons ici la preuve (on renvoie au texte pour les d\'efinitions et \'enonc\'es pr\'ecis) : Le crit\`ere de Parent et Yafaev proposition \ref{parentyafaev} ne peut \^etre appliqu\'e qu'\`a des couples de nombres premiers $(p,q)$ satisfaisant les conditions du cas non ramifi\'e de Ogg ( (2) proposition \ref{casdeogg}) et la condition suppl\'ementaire $p\equiv -1 \mod 3$. Nous g\'en\'eralisons ce th\'eor\`eme dans la proposition \ref{parentyafaevgeneralise} en supprimant cette derni\`ere condition. Celle-ci appara\^it dans un lemme de Parent et Yafaev sur le groupe des composantes de $\Jac(X^{pq}/w_q)_{\F_p}$. On consid\`ere $(\widetilde{{X^{pq}/w_q}})_{\F_p}$ la fibre en $p$ du mod\`ele r\'egulier de $X^{pq}/w_q$, obtenu par \'eclatement aux points singuliers du mod\`ele d\'ecrit par Cherednik et Drinfeld. Supposons que $p$ et $q$ satisfont les conditions du cas non ramifi\'e de Ogg, et que $q$ est \og assez grand\fg. Soit $\mathcal{J}$ la composante exceptionnelle de $(\widetilde{{X^{pq}/w_q}})_{\F_p}$ apparaissant apr\`es \'eclatement de l'unique point singulier d'\'epaisseur $2$, et $J\neq \mathcal{J}$ une autre composante irr\'eductible de $(\widetilde{{X^{pq}/w_q}})_{\F_p}$. 
Parent et Yafaev prouvent que si $p\equiv -1 \mod 3$, pour $p\gg q$, on a $(p+1)(J-\mathcal{J})\neq 0$ dans le groupe des composantes de $\Jac(X^{pq}/w_q)_{\F_p}$ (cf. \cite[lemma 3.1.3]{ParentYafaev}). Dans le troisi\`eme paragraphe nous g\'en\'eralisons ce lemme en montrant que le r\'esultat reste valable si $p\equiv 1 \mod 3$ (cf. lemme \ref{general}). 
 
Dans le quatri\`eme paragraphe, nous appliquons le crit\`ere de Parent-Yafaev (th\'eor\`eme \ref{parentyafaevgeneralise}) pour prouver le th\'eor\`eme \ref{resfinal}. Pour cela il suffit de construire un cycle sur le graphe $\mathscr{G}(X^{pq}_{\F_p}/w_q)$ constitu\'e de vecteurs de Gross contenant l'ar\^ete exceptionnelle de longueur $2$ avec une multiplicit\'e premi\`ere \`a $p$. Le chemin que nous construisons utilise \'egalement le vecteur d'Eisenstein qui appartient \`a l'espace engendr\'e par les vecteurs de Gross. Comme les vecteurs de Gross sont invariants par l'action de $w_q$, nous travaillons sur $\mathscr{G}(X^{pq}_{\F_p})$. 
Nous utilisons la description du graphe de $\mathscr{G}(X^{pq}_{\F_p})$ donn\'ee par Ribet. Le groupe des chemins $\mathcal{L}$ sur $\mathscr{G}(X^{pq}_{\F_p})$ est isomorphe au groupe des diviseurs sur les $p$-isog\'enies entre courbes elliptiques supersinguli\`eres en caract\'eristique $q$ \`a isomorphisme pr\`es. D'autre part le groupe des chemins $\mathcal{P}_S$ sur le graphe $\mathscr{G}(X_0(q)_{\F_{q}})$ est isomorphe au groupe des diviseurs sur les courbes elliptiques supersinguli\`eres en caract\'eristique $q$ \`a isomorphisme pr\`es. Le groupe des cycles sur $\mathscr{G}(X^{pq}_{\F_p})$ est l'intersection des noyaux de deux morphismes naturels $s_*$ et $t_*$ de $\mathcal{L}$ dans $\mathcal{P}_S$. 

Pour \'eviter la confusion entre les vecteurs de Gross sur le graphe $\mathscr{G}(X_0(q)_{\F_{q}})$ et les vecteurs de Gross sur le graphe $\mathscr{G}(X^{pq}_{\F_p})$, nous les d\'esignons respectivement par vecteurs de Gross-modulaire (cf. d\'efinition \ref{GrossModulaire}) et vecteurs de Gross-Shimura (cf. d\'efinition \ref{defvecteurdegross}). De m\^eme, le vecteur d'Eisenstein sur $\mathscr{G}(X_0(q)_{\F_{q}})$ et le vecteur d'Eisenstein sur $\mathscr{G}(X^{pq}_{\F_p})$ sont appel\'es respectivement vecteur d'Eisenstein-modulaire et vecteur d'Eisenstein-Shimura. 

Nous prouvons (sous certaines condition sur l'ordre $O_D$ de discriminant $D$) que l'image d'un vecteur de Gross-Shimura $\gamma_D\in \mathcal{L}$ par l'application $s_*$ (ou $t_*$) est le vecteur de Gross-modulaire $\Gamma_D\in \mathcal{P}_S$ correspondant au m\^eme discriminant. Construire un cycle $C$ constitu\'e de vecteurs de Gross-Shimura, revient \`a construire un chemin nul en vecteurs de Gross-modulaire. Pour trouver un cycle contenant les ar\^etes exceptionnelles avec une multiplicit\'e $\lambda$ premi\`ere \`a $p$, nous construisons un cycle de la forme $C_0=C+\lambda a_E$ o\`u $C$ est un chemin en vecteurs de Gross ne contenant pas l'ar\^ete exceptionnelle et $a_E$ est le vecteur d'Eisenstein-Shimura. Pour cela, nous nous ramenons \`a construire une combinaison lin\'eaire de vecteurs de Gross-modulaires \'egale au vecteur d'Eisenstein-modulaire en utilisant uniquement des vecteurs associ\'es \`a des  \og bons \fg\ ordres.

\section{Vecteurs de Gross}
Pour toute la suite, $p$ et $q$ sont deux nombres premiers distincts, tous deux $\ge 5$. Dans ce paragraphe, nous allons rappeler la notion de vecteurs de Gross sur le graphe dual de la fibre en $q$ de la courbe modulaire de niveau $q$, et sur le graphe dual de la fibre en $p$ d'une courbe de Shimura de discriminant $pq$. Nous invitons le lecteur \`a se r\'ef\'erer \`a \cite{Clarkthesis} et \cite{Ribet} pour une introduction aux courbes de Shimura, \`a \cite{Gross} pour la d\'efinition des vecteurs de Gross sur les courbes modulaires de niveau premier, et \`a \cite[\S 4]{ParentYafaev} pour le cas des vecteurs de Gross sur les courbes de Shimura de discriminant $pq$.

Dans les paragraphes suivants, on consid\'erera l'existence \'eventuelle de points rationnels du quotient de la courbe de Shimura $X^{pq}$ par l'involution d'Atkin-Lehner $w_q$. Cette existence de points rationnels ne peut se produire que dans les deux cas d\'ecrits par Ogg dans la proposition suivante. 
\begin{prop}
\label{casdeogg}
Si $X^{pq}/w_q(\Q)$ est non vide, alors $\legendre{p}{q}=-1$, et une des deux conditions suivante est satisfaite
\begin{itemize}
\item[(1)] $p\equiv 3 \mod 4$; (cas ramifi\'e)
\item[(2)] $p\equiv 1 \mod 4$ et $q\equiv 3 \mod 4$ (cas non ramifi\'e). 
\end{itemize}
\end{prop}

\begin{proof}
Consulter \cite[thm 3.1]{RSY}. 
\end{proof}
Nous nous placerons pour ce travail dans le cas non ramifi\'e.

\subsection{Vecteurs de Gross sur le graphe modulaire}

D'apr\`es le travail de Deligne et Rapoport (cf. \cite[th\'eor\`eme 6.9 et exemple 6.16]{DR}) la fibre en $q$ de la courbe modulaire $X_0(q)$ est constitu\'ee de deux composantes irr\'eductibles $s_1$ et $s_2$ isomorphes \`a $\Pro^1(\F_q)$ se coupant transversalement. Il d\'ecoule de cette description que le graphe dual de la fibre en $q$ de la courbe modulaire $X_0(q)$, not\'e $\mathscr{G}((X_0(q))_{\F_q})$, est constitu\'e de deux sommets $s_1$ et $s_2$ et d'ar\^etes entre $s_1$ et $s_2$. Ces ar\^etes correspondent \`a l'ensemble $S$ des classes d'isomorphismes de courbes elliptiques supersinguli\`eres sur $\overline{\F}_{q}$. On note $E_{j_1},\dots, E_{j_{g+1}}$ un syst\`eme de repr\'esentants de $S$, o\`u les $j_k$ parcourent le sous-ensemble de $\F_{q^2}$ constitu\'e par les $j$-invariants supersinguliers, et $g$ est le genre de $X_0(q)$. 
Soit $B_{q\infty}$ l'alg\`ebre de quaternions sur $\Q$ ramifi\'ee en $q$ et l'infini. \`A la classe d'isomorphismes de la courbe elliptique $E_{j_k}$, on associe l'ordre maximal $\Omega_k=\End(E_{j_k}) \subseteq B_{q\infty}$.

Le cardinal de $(\Omega_k^*/\{\pm 1\})$ est appel\'e la longueur de l'ar\^ete de $\mathscr{G}((X_0(q))_{\F_q})$ correspondant \`a $E_{j_k}$. D'apr\`es \cite[theorem 10.1 p. 103]{Silv2}, si $j_k\neq 0,1728$ ce cardinal est \'egal \`a 1. Il y a donc au plus deux ar\^etes de longueur $>1$.

On d\'esigne par $\mathcal{P}_S$ l'ensemble des chemins sur le graphe $\mathscr{G}((X_0(q))_{\F_q})$, c'est le $\Z$-module des diviseurs sur $S$. On d\'efinit le degr\'e d'un chemin par : $\text{degr\'e}(\sum_{k=1}^{g+1}\lambda_{j_k}E_{j_k})=\sum_{k=1}^{g+1} \lambda_{j_k}. $
 On note $\mathcal{P}_S^0$ le $\Z$-module des diviseurs de degr\'e $0$ sur $S$. La $\Z$-alg\`ebre $\T_{\Gamma_0(q)}$ engendr\'ee par les op\'erateurs de Hecke $T_l$, pour $l\neq q$ premier, agit sur les groupes $\mathcal{P}_S$ et $\mathcal{P}_S^0$ par l'action $T_l(E)=\sum_{C_l\subset E}E/C_l$, o\`u $E\in S$ et $C_l$ parcourt l'ensemble des sous-groupes d'ordre $l$ de $E$. 
\begin{dfn}
L'accouplement de monodromie (Cf. \cite[Expos\'e 9 \S 9]{SGA7-I}) est la forme bilin\'eaire non d\'eg\'en\'er\'ee $\coupl{\cdot,\cdot}$  sur $\mathcal{P}_S$ d\'efinie par : 
$$\coupl{E_{j_k},E_{j_l}}=\delta_{j_k,j_l}\card(\Omega_k^*/\{\pm 1\}),$$
o\`u $\delta_{j_k,j_l}$ est le symbole de Kronecker. Ainsi $E_{j_1}, \dots E_{j_{g+1}}$ est une base orthogonale pour cet accouplement. Chaque $E_{j_k}$ est de norme $1$ pour cet accouplement, sauf si $j_k=0$ ou $1728$. 
\end{dfn}

\begin{dfn}
\label{EisensteinModulaire}
Le vecteur d'Eisenstein (ou vecteur d'Eisenstein-modulaire) $A_E\in \frac{1}{6}\mathcal{P}_S$ est : 
$$A_E=\sum_{k=1}^{g+1} \frac{1}{\card(\Omega_k^*/\{\pm 1\})}E_{j_k}.$$
\end{dfn}
Le vecteur d'Eisenstein $A_E$ forme une base de $(\mathcal{P}_S^0)^{\bot}$, l'espace orthogonal \`a $\mathcal{P}_S^0$ pour cet accouplement. Il est choisi de telle sorte que, pour $v\in \mathcal{P}_S$, le degr\'e de $v$ est $\coupl{v,A_E}$.

\begin{dfn}
\label{GrossModulaire}
Soit $O_{D}$ l'ordre quadratique imaginaire de discriminant $D$. On lui associe un vecteur $\Gamma_D\in\frac{1}{12}\mathcal{P}_S$, appel\'e \emph{vecteur de Gross sur le graphe modulaire} (ou vecteur de Gross-modulaire) : 
$$\Gamma_D=\frac{1}{2u(D)}\sum_{k=1}^{g+1} H_k(D)E_{j_k},$$
o\`u $H_k(D)$ est le nombre de plongements optimaux de $O_D$ dans $\Omega_k$ modulo conjugaison par $\Omega_k^*$ (cf. \cite[\S 1, p 122]{Gross}), et $u(D)=\card(O_D^*/\{ \pm 1\})$. On a $u(D)=1$, sauf pour $D=-3$ ou $-4$. De plus on a $u(-3)=3$, $u(-4)=2$. 
\end{dfn}
Par la formule des traces d'Eichler, lorsque $q$ est scind\'e dans $O_D$ ou lorsque $q$ divise le conducteur de $O_D$, il n'existe pas de plongement de $O_D$ dans $B_{q\infty}$ (voir \cite[\S 1, p 122]{Gross}). Le vecteur $\Gamma_D$ est non nul si $q$ est inerte ou ramifi\'e dans $O_D$ et $q^2$ ne divise pas $D$. La proposition suivante donne une interpr\'etation des vecteurs de Gross en terme de r\'eduction modulo $q$ des courbes elliptiques d\'efinies sur $\overline{\Q}$ ayant multiplication complexe par $O_D$. 
Nous nous pla\c cons dans le cas o\`u $q$ est inerte dans $O_D$, car nous ne manipulerons dans la suite que de tels vecteurs de Gross. 
\begin{prop}
\label{vecteurdegrossmodulaire}
Soit $O_D$ un ordre quadratique imaginaire de discriminant $D$ et de nombre de classe $h(D)$, tel que $q$ soit inerte dans $O_D$. Fixons $\mathcal{Q}$ une place de $\overline{\Q}$ au dessus de $q$. Le vecteur $u(D)\Gamma_D$ est constitu\'e de la somme des r\'eductions (non forc\'ement distinctes) modulo $\mathcal{Q}$ des $h(D)$ courbes elliptiques sur $\overline{\Q}$ ayant multiplication complexe par $O_D$. 
\end{prop}
\begin{proof}
 Par la formule des traces d'Eichler (cf. \cite[thm 5.11, p 92]{Vigneras}), le nombre de plongements optimaux de $O_D$ dans les ordres maximaux de $B_{q\infty}$ est 
$$\sum_{i=1}^{g+1} H_i(D)=\bigg(1-\legendre{D}{q}\bigg)h(D)=2h(D).$$

D'autre part, d'apr\`es le th\'eor\`eme fondamental de la multiplication complexe, il existe $h(D)$ courbes elliptiques $E_1/\overline{\Z}$,$E_2/\overline{\Z}$,\dots,$E_{h(D)}/\overline{\Z}$ \`a $\overline\Q$-isomorphisme pr\`es ayant multiplication complexe par $O_D$. Pour chacune de ces courbes elliptiques, il existe deux isomorphismes conjugu\'es $f_1,f_2\colon O_D\rightarrow \End(E)$. La place $\mathcal{Q}$ d\'efinit une inclusion de $\overline{\Z}$ dans une cl\^oture alg\'ebrique de $O_D\otimes \Z_q$. Soit $W$ un anneau de valuation discr\`ete complet contenant $O_D\otimes \Z_q$, de corps r\'esiduel $\overline{\F}_q$, et tel que $E_{1},\dots, E_{h(D)}$ sont d\'efinies sur $W$. On note $q'$ l'uniformisante de $W$. On consid\`ere l'ensemble des couples $(E,g)$ constitu\'es d'une courbe elliptique $E/W$ \`a multiplication complexe par $O_D$ sur $W$, muni d'un isomorphisme fix\'e $g\colon O_D\rightarrow \End(E)$. On dit que les couples $(E,g)$ et $(E',g')$ sont \'equivalents s'il existe un isomorphisme $i\colon (E\mod q')$ $\rightarrow (E' \mod q')$ tel que $(g'(\alpha)\mod q')\circ i=i\circ (g(\alpha)\mod q')$ pour tout $\alpha\in O_D$. D'apr\`es Gross \cite[fin du \S 2 pp. 128-129]{Gross}, les plongements optimaux de $O_D$ dans les ordres maximaux de $B_{q\infty}$ correspondent aux couples $(E,g)$, \`a \'equivalence pr\`es. Le plongement optimal associ\'e \`a un tel couple $(E,g)$ est obtenu par composition de $g\colon O_D\rightarrow\End(E)$ et de la r\'eduction $\End(E)\rightarrow \End(E\mod q')$. Toute courbe elliptique $E/W$ \`a multiplication complexe par $O_D$ est une tordue d'une des courbes $E_{1},\dots E_{h(D)}$. Donc tout couple $(E,g)/W$ est \'equivalent \`a un des couples $(E_l,f)/W$. Ainsi chacun des $2h(D)$ plongement optimaux de $O_D$ dans les ordres maximaux de $B_{q\infty}$ correspond \`a d'un des $2h(D)$ couples $(E_l,f)/\overline{\Z}$. Cette correspondance est bijective par \'egalit\'e des cardinaux. 
Pour conclure il suffit de voir que : 
$$2u(D)\Gamma_D=\sum_{k=1}^{g+1} \sum_{\topbot{O_D \rightarrow\Omega_k}{\text{optimal}}} E_{j_k}. $$
Donc on a 
$$2u(D)\Gamma_D=\sum_{\topbot{(E,g)/W} {\text{\`a \'equiv. pr\`es}}} (E \mod q') =2(E_1+\dots +E_{h(D)}) \mod \mathcal{Q}.$$

\end{proof}

\begin{prop}
\label{epaisseur}
Toute ar\^ete est de longueur $1$ sauf dans les deux cas suivants : 
\begin{itemize}
\item si $q$ est inerte dans $O_{-4}$, alors $1728 \mod q$ est un $j$-invariant supersingulier et l'ar\^ete associ\'ee \`a la courbe elliptique $E_{1728}$ de $j$-invariant $1728 \mod q$ est de longueur $2$. Plus pr\'ecis\'ement $O_{-4}=\Z[\zeta_4]$ se plonge dans $\End{E_{1728}}$. 
\item Si $q$ est inerte dans $O_{-3}$, alors $0 \mod q$ est un $j$-invariant supersingulier et l'ar\^ete associ\'ee \`a la courbe elliptique $E_0$ de $j$-invariant $0$ est de longueur $3$. Plus pr\'ecis\'ement $O_{-3}=\Z[\zeta_6]$ se plonge dans $\End{E_{0}}$. 
\end{itemize}
\end{prop}

\begin{proof}
D\'ecoule de \cite[thm 10.1 p. 103]{Silv1} et du fait qu'une ar\^ete de longueur $2$ (respectivement de longueur $3$) correspond \`a un ordre maximal de $B_{q\infty}$ dans lesquels se plonge $O_{-4}$ (respectivement $O_{-3}$). 

\end{proof}

L'alg\`ebre de Hecke $\T_{\Gamma_0(q)}$ agit sur l'espace $S_2(\Gamma_0(q))$ des formes modulaires paraboliques primitives $f$ de poids $2$ sur $\Gamma_0(q)$. On d\'esigne par $I_e$ l'id\'eal d'enroulement de $\T_{\Gamma_0(q)}$, c'est-\`a-dire l'ensemble des op\'erateurs de $\T_{\Gamma_0(q)}$ annulant toutes les formes modulaires de $S_2(\Gamma_0(q))$ telles que $L(f,1)\neq 0$. Pour une forme modulaire primitive $f$, la formule de Gross \cite[Cor. 11.6, p. 167]{Gross} et de fa\c con plus g\'en\'erale un th\'eor\`eme de Waldspurger \cite[p 397]{LR}, relie la valeur de $L(f,1)$ \`a la norme de l'image $\Gamma_{f,D}$ du vecteur de Gross $\Gamma_{D}$ par l'idempotent primitif $1_f\in \T$ associ\'e \`a $f$. On en d\'eduit la proposition suivante. 

\begin{prop}
\label{GrossApplication}
L'espace engendr\'e par les projections orthogonales des vecteurs de Gross de discriminant $D$ premier \`a $q$ sur $\mathcal{P}_S\otimes \Q$ est $\mathcal{P}_S[I_e] \otimes \Q$. 
\end{prop}
\begin{proof}
Consulter \cite[prop 4,2]{Parent}. 
\end{proof}
On rappelle finalement que le vecteur d'Eisenstein appartient au $\Q$-espace engendr\'e par les vecteurs de Gross. Cela d\'ecoule de la proposition \ref{Eisenstein} ci-dessous. 
\begin{prop}
\label{Eisenstein}
Soient $D$ le discriminant de l'ordre maximal $O_D$ d'un corps quadratique imaginaire, tel que $q$ soit inerte ou ramifi\'e dans $O_D$ et $l\neq q$ un nombre premier. On a : 
$$\lim_{n \rightarrow \infty} \frac{\coupl{A_E,A_E}}{\coupl{A_E,\Gamma_{l^{2n}D}}} \Gamma_{l^{2n}D}=A_E.$$
Il en d\'ecoule que $A_E$ appartient au $\Q$-espace vectoriel engendr\'e par les $\Gamma_{l^{2n}D}$, pour $n\ge 1$. 
\end{prop}
\begin{proof}
Consulter \cite[thm 4.3]{Parent}. 
\end{proof}
On peut remarquer de la m\^eme fa\c con que le vecteur d'Eisenstein-Shimura sur le graphe dual de la fibre en $p$ de la courbe de Shimura $X^{pq}$ (cf. D\'efinition \ref{EisensteinShimura}) s'exprime lui aussi comme combinaison lin\'eaire de vecteurs de Gross.

\subsection{Vecteurs de Gross sur le graphe de Shimura}

Soit $X^{pq}$ la courbe de Shimura de discriminant $pq$. Ribet donne la description de $\mathscr{G}(X^{pq}_{\F_p})$, le graphe dual de la fibre en $p$ de cette courbe de Shimura, dans \cite[Propositions 4.4 et 4.7]{Ribet}. L'ensemble de ses sommets est constitu\'e de deux copies $S_1\sqcup S_2$ de l'ensemble $S$ des classes d'isomorphismes de courbes elliptiques supersinguli\`eres $\overline{\F}_{q}$. L'ensemble de ses ar\^etes est l'ensemble $\{e_1,\dots,e_n\}$ des classes d'isomorphismes de $p-$isog\'enies entre ces courbes elliptiques. Chacune de ces isog\'enies $e_i \colon E_{j_1}\rightarrow E_{j_2}$ s'interpr\`ete comme une ar\^ete reliant le sommet de $S_1$ correspondant \`a $E_{j_1}$ au sommet de $S_2$ correspondant \`a $E_{j_2}$. De mani\`ere \'equivalente l'ensemble des ar\^etes du graphe s'identifie avec l'ensemble des couples $(E,C_p)$, o\`u $E$ est une courbe elliptique supersinguli\`ere \`a isomorphisme pr\`es sur $\overline{\F}_{q}$, et $C_p$ est un sous-groupe d'ordre $p$ de $E$ d\'efini \`a action de $\End(E)^*$ pr\`es. On munit ce graphe d'une orientation en consid\'erant que toutes les ar\^etes sont orient\'ees d'un sommet de $S_1$ vers un sommet de $S_2$. \`A chaque ar\^ete $e=(E,C_p)$ on associe l'ordre d'Eichler $\End(e)=\{\alpha\in\End(E) \ : \ \alpha(C_p)\subseteq C_p\}$ de niveau $p$ dans $B_{q\infty}$. On d\'efinit la longueur d'une ar\^ete $e=(E,C_p)$ comme le cardinal de $\End(e)^*/\{\pm 1\}$. Comme pour le graphe modulaire, les ar\^etes ont toutes une longueur 1, sauf peut-\^etre les quatre ar\^etes d\'ecrites dans le corollaire \ref{epaisseur2} ci-apr\`es.

Les involutions d'Atkin Lehner $w_p$ et $w_q$ agissent sur le graphe de la mani\`ere suivante: l'involution $w_p$ \'echange chaque sommet de $S_1$ (respectivement $S_2$) avec le sommet de $S_2$ (respectivement $S_1$) correspondant \`a la m\^eme courbe elliptique \`a isomorphisme pr\`es, et \'echange l'ar\^ete correspondant \`a une isog\'enie $(E,C_p)$ avec l'oppos\'e de l'ar\^ete correspondant \`a son isog\'enie duale $(E/C_p,E[p]/C_p)$. L'involution $w_q$ agit comme le Frobenius sur le graphe, c'est-\`a-dire qu'elle \'echange chaque sommet de $S_1$ (respectivement $S_2$) avec le sommet de $S_1$ (respectivement $S_2$) correspondant \`a l'image par le Frobenius en $q$ de la courbe elliptique correspondante, et \'echange l'ar\^ete correspondant \`a l'isog\'enie $(E,C_p)$ avec l'ar\^ete correspondant \`a l'image de cette isog\'enie par le Frobenius. 

Comme dans le paragraphe pr\'ec\'edent, on consid\`ere $\mathcal{L}$ le $\Z$-module des diviseurs sur les ar\^etes du graphe $\mathscr{G}(X^{pq}_{\F_p})$. On appelle les \'el\'ements de $\mathcal{L}$ les chemins sur le graphe. Notons $\mathcal{L}^0$ le $\Z$-module des diviseurs de degr\'e $0$. 
La $\Z$-alg\`ebre $\T_{\Gamma_0(pq)}$ engendr\'ee par les op\'erateurs de Hecke $T_l$, pour $l\neq p,q$ premier, agit sur ces groupes par l'action : 
$$T_l((E,C_p))=\sum_{C_l\subset E}(E/C_l,(C_p+C_l)/C_l),$$
 o\`u $C_l$ parcourt les sous-groupes d'ordre $l$ de $E$. 

On d\'efinit l'accouplement de monodromie sur $\mathcal{L}$ par : 
$$\coupl{e_i,e_j}=\delta_{i,j}\card(\End(e_i)^*/\{\pm 1\}).$$
On remarque que les $(e_i)_{i=1\dots n}$ forment une base orthogonale de $\mathcal{L}$ pour cet accouplement, et que presque tous les $e_i$ sont de norme $1$. 

\begin{dfn}
\label{EisensteinShimura}
Le vecteur d'Eisenstein (ou vecteur d'Eisenstein-Shimura) est d\'efini par : 
$$a_E=\sum_{i=1}^{n} \frac{1}{\card(\End(e_i)^*/\{\pm 1\})}e_i.$$
Il forme une base de $(\mathcal{L}^0)^{\bot}$, l'espace orthogonal \`a $\mathcal{L}^0$ pour l'accouplement de monodromie. 
\end{dfn}
On remarque que le module $\mathcal{L}$ s'identifie naturellement avec le groupe des diviseurs des points supersinguliers de $X_0(pq)(\F_{q^2})$. Le vecteur d'Eisenstein sur $\mathscr{G}(X^{pq}_{\F_p})$ correspond donc \`a un vecteur d'Eisenstein pour la courbe modulaire de niveau non premier $X_0(pq)_{\F_q}$. De la m\^eme fa\c con, les vecteurs de Gross sur $\mathscr{G}(X^{pq}_{\F_p})$ qui seront introduits dans la d\'efinition \ref{defvecteurdegross} ci-dessous, correspondent \`a des vecteurs de Gross pour la courbe modulaire $X_0(pq)_{\F_q}$.

\begin{dfn}
On note par $s$ l'application qui a une ar\^ete $e=(E_{j_k},C_p)$ de $\mathscr{G}(X^{pq}_{\F_p})$ associe la courbe $E_{j_k}\in S$.
On a :  
$$s_*\colon \mathcal{L} \rightarrow \mathcal{P}_S$$ 
$$\sum_{i=1}^n \lambda_i e_i\mapsto \sum_{i=1}^n \lambda_i s(e_i). $$
De mani\`ere similaire on d\'efinit l'application $t$ qui a l'ar\^ete $e=(E_{j_k},C_p)$ associe la courbe $E_{j_k}/C_p\in S$.
On a : 
$$t_*\colon \mathcal{L} \rightarrow \mathcal{P}_S$$ 
$$\sum_{i=1}^n \lambda_i e_i\mapsto \sum_{i=1}^n \lambda_i t(e_i). $$
\end{dfn}
Soit $Y$ l'intersection des noyaux de $s_*$ et $t_*$. C'est un sous $\Z$-module de $\mathcal{L}^0$. D'apr\`es \cite[corollary 4.5]{Ribet}, $Y$ est isomorphe \`a $H_1(\mathscr{G}(X^{pq}_{\F_p}),\Z)$. On appelle les \'el\'ements de $Y$ les cycles ou chemins ferm\'es de $\mathscr{G}(X^{pq}_{\F_p})$.

\begin{dfn}
\label{defvecteurdegross}
Soit $O_{D}$ un ordre quadratique imaginaire de discriminant $D$. On d\'efinit le \emph{vecteur de Gross sur le graphe dual de la fibre en $p$ de la courbe de Shimura}, (ou vecteur de Gross-Shimura), $\gamma_D\in\frac{1}{12}\mathcal{L}$, par : 
$$ \gamma_D=\sum_{i=1}^n \frac{h_i(D)}{\card(\End(e_i)^*/\{\pm 1\})} e_i,$$
o\`u $h_i(D)$ est le nombre de plongements optimaux de $O_D$ dans l'ordre d'Eichler $R_i=\End(e_i)$. 
\end{dfn}

Lorsque $q$ est scind\'e (ou $p$ est inerte) dans l'ordre $O_D$, d'apr\`es la formule des traces d'Eichler, il n'existe aucun plongement de $O_D$ dans les ordres d'Eichler de niveau $p$ de $B_{q\infty}$, et le vecteur $\gamma_D$ est nul (cf. \cite[thm 5.11, p 92]{Vigneras}). Lorsque $q$ est inerte (ou ramifi\'e et ne divisant pas le conducteur de $O_D$) et $p$ scind\'e (ou ramifi\'e et ne divisant pas le conducteur de $O_D$) dans $O_D$, le vecteur $\gamma_D$ est non nul. De mani\`ere similaire \`a la proposition \ref{vecteurdegrossmodulaire}, la proposition suivante donne une description des ar\^etes du vecteur $\gamma_D$ en terme de r\'eduction modulo $q$ des isog\'enies entre courbes elliptiques sur $\overline{\Q}$ \`a multiplication complexe par $O_D$. Nous nous restreignons au cas o\`u $q$ est inerte et $p$ scind\'e dans $O_D$, car nous ne manipulerons dans la suite que des vecteurs satisfaisant ces hypoth\`eses.

\begin{prop}
\label{vecteurdegrossshimura0}
Soit $O_D$ un ordre quadratique imaginaire de discriminant $D$ et de nombre de classe $h(D)$, tel que $q$ soit inerte et $p$ scind\'e dans $O_D$. Fixons $\mathcal{Q}$ une place de $\overline{\Q}$ au dessus de $q$. Le vecteur $\gamma_D$ est constitu\'e de $2h(D)$ ar\^etes, chacune de ces ar\^etes $e$ pond\'er\'ee par $2/(\card(\End(e)^*/\{\pm 1\}))$. Ces ar\^etes correspondent aux r\'eductions modulo $\mathcal{Q}$ des $2h(D)$ $p$-isog\'enies entre les $h(D)$ courbes elliptiques sur $\overline{\Q}$ ayant multiplication complexe par $O_D$.
\end{prop}

\begin{proof}
Par la formule des traces d'Eichler, le nombre de plongements optimaux de $O_D$ dans les ordres d'Eichler de niveau $p$ de $B_{q\infty}$ est 
$$\sum_{i=1}^n h_i(D)=\bigg(1-\legendre{D}{q}\bigg)\bigg(1+\legendre{D}{p}\bigg)h(D)=4h(D).$$ 
Les ordres d'Eichler $R_1,\dots, R_n$ de niveau $p$ correspondent aux $\End(E_{j_i}, C_p)$, o\`u $i=1,\dots, g+1$ et $C_p$ parcourt les sous-groupes d'ordre $p$ de $E_{j_i}$ \`a isomorphisme pr\`es. 

Soit $\Phi \colon O_D\rightarrow\End(E_{j_i},C_p)$ un plongement optimal. Le plongement induit $\Phi'\colon O_D\rightarrow\End(E_{j_i})$ est \'egalement optimal. D'apr\`es la proposition \ref{vecteurdegrossmodulaire}, ce plongement $\Phi'$ provient d'un couple $(E,g)$, o\`u $E/\overline{\Q}$ est une courbe elliptique \`a multiplication complexe par $O_D$ munie d'un isomorphisme $g\colon O_D\rightarrow \End(E)$, de telle sorte que $(E \mod \mathcal{Q})=E_{j_i}$. Comme la r\'eduction modulo $\mathcal{Q}$ r\'ealise un isomorphisme $E[p]\rightarrow E_{j_i}[p]$, il existe un sous-groupe $C$ d'ordre $p$ de $E$ tel que $(C \mod \mathcal{Q})=C_p$. Pour des raisons de compatibilit\'e, on a $g(O_D)$ inclus dans $\End(E,C)$, c'est-\`a-dire $\End(E)=\End(E,C)\simeq O_D$. L'isog\'enie $E\rightarrow E/C$ de noyau $C$ est une $p$-isog\'enie entre $E$ et une courbe elliptique $E/C$ \`a multiplication complexe par $O_D$. Par la th\'eorie de la multiplication complexe, une telle isog\'enie correspond \`a un id\'eal $P$ de $O_D$ au dessus de $p$. Comme $p$ est scind\'e dans $O_D$, il existe deux isog\'enies distinctes de $E$ dans des courbes elliptiques \`a multiplication complexe par $O_D$, et $C$ est le noyau d'une de ces deux isog\'enies. 
 Finalement un plongement optimal  $\Phi$ de $O_D$ dans un ordre d'Eichler $R=\End(E_{j_i},C_p)$ correspond \`a un des $4h(D)$ couples $((E,C),g)$ constitu\'es d'une $p$-isog\'enie \`a isomorphisme pr\`es $(E,C)$ entre courbes elliptiques sur $\overline{\Q}$ ayant multiplication complexe par $O_D$, et d'un isomorphisme $g\colon O_D \rightarrow \End(E,C)$. Ici $\Phi$ est la compos\'ee de $g$ et de la r\'eduction modulo $\mathcal{Q}$ de $\End(E,C) \rightarrow \End(E_{j_i},C_p)$. Cette correspondance est bijective par \'egalit\'e des cardinaux.

\end{proof}

La proposition suivante d\'ecrit l'action des involutions d'Atkin-Lehner sur les vecteurs de Gross. 
\begin{prop}
\label{actionatkinlehner}
Avec les hypoth\`eses de la proposition \ref{vecteurdegrossshimura0}, on a $w_p(\gamma_D)=-\gamma_D$, et $w_q(\gamma_D)=\gamma_D$. 
\end{prop}
\begin{proof} Prouvons d'abord que $w_p(\gamma_D)=-\gamma_D$. Soit $\tau$ une permutation de l'ensemble $\{1,\dots, n\}$ telle que, pour chaque ar\^ete $e_i$, l'ar\^ete $e_{\tau(i)}$ est son isog\'enie duale. L'involution $w_p$ agit sur les ar\^etes du graphe en envoyant $e_i$ sur $-e_{\tau(i)}$. Ainsi on a 
$$ w_p(\gamma_D)=\sum_{i=1}^{n}\frac{h_i(D)}{\card(\End(e_i)^*/\{\pm 1\})} (-e_{\tau(i)}),$$
or $\End(e_i)\simeq \End(e_{\tau(i)})$ donc $w_p(\gamma_D) =-\gamma_D.$

Rappelons que $w_q$ agit comme le Frobenius sur les sommets et les ar\^etes du graphe. Soit $\sigma \in \Gal(\overline{\Q}/\Q)$ tel que $\sigma$ se r\'eduit en le Frobenius de $\Gal(\overline{\F}_q/\F_q)$. Consid\'erons $\mathcal{I}_{\overline{\Q}}$ l'ensemble des $p$-isog\'enies \`a isomorphismes pr\`es entre courbes elliptiques sur $\overline{\Q}$ \`a multiplication complexe par $O_D$. D'apr\`es la proposition \ref{vecteurdegrossshimura0}, 
$$\gamma_D=\sum_{g\in \mathcal{I}_{\overline{\Q}}} \frac{1}{\card(\End(g \mod q)^*/\{\pm 1\})} (g\mod q),$$
donc 
$$w_q(\gamma_D)=\sum_{g\in \mathcal{I}_{\overline{\Q}}} \frac{1}{\card(\End(g \mod q)^*/\{\pm 1\})} (\sigma(g)\mod q).$$
Or  $\sigma$ r\'ealise une bijection de $\mathcal{I}_{\overline{\Q}}$ dans $\mathcal{I}_{\overline{\Q}}$ telle que pour tout $g\in \mathcal{I}_{\overline{\Q}}$, on a $\End(\sigma(g) \mod q)\simeq\End(g \mod q)$ donc $w_q(\gamma_D)=\gamma_D$.
\end{proof}

La proposition \ref{vecteurdegrossshimura0} nous permet de d\'ecrire l'image de $\gamma_D$ par les applications $s_*$ et $t_*$, en comparant le vecteur $\gamma_D$ sur $\mathscr{G}(X^{pq}_{\F_p})$ au vecteur $\Gamma_D$ sur $\mathscr{G}((X_0(q))_{\F_q})$. 
\begin{cor}
\label{vecteurdegrossshimura}
Avec les hypoth\`eses de la proposition \ref{vecteurdegrossshimura0}, les ar\^etes du vecteur de Gross $\gamma_D$ sur le graphe $\mathscr{G}(X^{pq}_{\F_p})$ sont distribu\'ees de la fa\c con suivante : de chaque sommet $E_{j_k}$ dans $S_1$ partent $H_k(D)$ ar\^etes de $\gamma_D$. Chacune de ces ar\^etes est pr\'ec\'ed\'ee du coefficient $2/l$, o\`u $l$ est la longueur de l'ar\^ete. 
En particulier, si le vecteur $\gamma_D$ ne contient pas d'ar\^etes de longueur $2$ ou $3$, on a : 
$$s_*(\gamma_D)=t_*(\gamma_D)=4\Gamma_D.$$ 
\end{cor}

\begin{proof}
Dans la d\'emonstration de la proposition \ref{vecteurdegrossshimura0}, nous avons montr\'e que les $H_k(D)$ plongements optimaux de $O_D$ dans l'ordre maximal $\End(E_{j_k})$ d\'efinissent $2H_k(D)$ plongements optimaux de $O_D$ dans des ordres d'Eichler de la forme $\End(E_{j_k},C_p)$ qui correspondent \`a des ar\^etes de $\mathscr{G}(X^{pq}_{\F_p})$ partant du sommet $E_{j_k}\in S_1$. Cela se traduit par la formule suivante : 
$$\sum_{\topbot{i=1...n}{e_i=(E_{j_k},C_p)}} h_i(D)=2H_k(D).$$
Si le vecteur $\gamma_D$ ne contient pas d'ar\^etes de longueur $2$ ou $3$, on a : 
$$ \gamma_D=\sum_{i=1}^n h_i(D)e_i.$$
On en d\'eduit que 
$$s_*(\gamma_D)=\sum_{k=1}^{g+1}2H_k(D)E_{j_k}=4\Gamma_D.$$
D'autre part pour chaque ar\^ete $e=(E_{j_k},C_p)$, on a $$s_*(-w_p(e))=s_*((E_{j_k}/C_p,E_{j_k}[p]/C_p))=E_{j_k}/C_p=t_*(e).$$ Ce qui implique que $t_*(\gamma_D)=s_*(-w_p(\gamma_D))$. Or d'apr\`es la proposition \ref{actionatkinlehner}, on a $\gamma_D=-w_p(\gamma_D)$. En cons\'equence $t_*(\gamma_D)=s_*(\gamma_D)$, ce qui ach\`eve la preuve de la proposition. 

\end{proof}

En appliquant la proposition \ref{vecteurdegrossshimura0}, on d\'ecrit les ar\^etes de longueur $\neq 1$. 
\begin{cor}
\label{epaisseur2}
Toutes les ar\^etes du graphe $\mathscr{G}(X^{pq}_{\F_p})$ sont de longueur $1$ sauf dans les cas suivants : 
\begin{itemize}
\item Si $q$ est inerte et $p$ scind\'e dans $O_{-4}$, il existe exactement deux ar\^etes de longueur $2$. Ces deux ar\^etes relient le sommet $E_{1728}\in S_1$ au sommet $E_{1728}\in S_2$ et sont \'echang\'ees par $w_q$. 
\item Si $q$ est inerte et $p$ scind\'e dans $O_{-3}$, il existe exactement deux ar\^etes de longueur $3$. Ces deux ar\^etes relient le sommet $E_{0}\in S_1$ au sommet $E_{0}\in S_2$ et sont \'echang\'ees par $w_q$. 
\end{itemize}
En particulier si on se place dans le cas non ramifi\'e de Ogg (voir proposition \ref{casdeogg}), il existe deux ar\^etes exceptionnelles de longueur $2$ sur $\mathscr{G}(X^{pq}_{\F_p})$. Donc une seule ar\^ete exceptionnelle de longueur $2$ sur $\mathcal{G}((X^{pq}/w_q)_{\F_p})$. 
\end{cor}

\begin{proof}
 
Comme dans la proposition \ref{epaisseur}, on voit que les ar\^etes de longueur $2$ sont les ar\^etes de $\gamma_{-4}$ et les ar\^etes de longueur $3$ sont les ar\^etes de $\gamma_{-3}$. 

D'apr\`es la formule des traces d'Eichler, si $q$ est scind\'e ou $p$ est inerte dans $O_{-4}$, le vecteur $\gamma_{-4}$ est nul. Supposons que $q$ est inerte et $p$ scind\'e dans $O_{-4}$. Soit $E/\overline{\Q}$ l'unique courbe elliptique ayant multiplication par $O_{-4}$. On sait que le $j$-invariant de $E$ est $1728$. D'apr\`es la proposition \ref{vecteurdegrossshimura0}, les ar\^etes de $\gamma_{-4}$ correspondent aux r\'eductions modulo $q$ des $2$ $p$-isog\'enies entre courbes elliptiques \`a multiplication complexe par $O_{-4}$.
 Soit $\sigma\in \Gal(\overline{\Q} / \Q)$ qui se r\'eduit en le Frobenius modulo $q$. Le morphisme $\sigma$ agit comme la conjugaison complexe sur $O_{-4}$, car $q$ est inerte dans $O_{-4}$. Comme $p$ est scind\'e dans $O_{-4}$, l'id\'eal $pO_{-4}$ se d\'ecompose en le produit de deux id\'eaux $P$ et $\sigma(P)$ dans $O_{D}$. Par la th\'eorie de la multiplication complexe, ces id\'eaux correspondent \`a deux $p$-isog\'enies conjugu\'ees par $\sigma$, not\'ees $[P]$ et $[\sigma(P)]$, de $E\rightarrow E$. Ces deux $p$-isog\'enies se r\'eduisent en deux ar\^etes du graphe $\mathscr{G}(X^{pq}_{\F_p})$ reliant le sommet $E_{1728}\in S_1$ au sommet $E_{1728}\in S_2$. Ces deux ar\^etes sont \'echang\'ees par $w_q$ qui agit comme le Frobenius sur le graphe, ce qui nous donne la forme annonc\'ee de l'ensemble des ar\^etes de $\gamma_{-4}$. 

On traite le cas des ar\^etes de longueur $3$ de mani\`ere similaire. 

\end{proof}

\begin{prop}
\label{intersection}
Soient $O_{D_1}$ et $O_{D_2}$ deux ordres quadratiques imaginaires de discriminant respectifs $D_1$ et $D_2$ premiers \`a $q$.  Supposons aussi que $q$ soit inerte dans $O_{D_1}$ et $O_{D_2}$. Il existe une borne $B$ d\'ependant de $q$ telle que, pour $p\ge B$, si $p$ est scind\'e dans ces deux ordres, les vecteurs de Gross $\gamma_{D_1}$ et $\gamma_{D_2}$ sur le graphe dual de la fibre en $p$ de la courbe de Shimura $X^{pq}$ sont sans ar\^etes communes. 
\end{prop}
\begin{proof}

Soient $\alpha_1$ et $\alpha_2$ tels que $O_{D_i}=\Z[\alpha_i]$, pour $i = 1, 2$. 
Soient $\phi_i \colon O_{D_i}\rightarrow \Omega$ des plongements optimaux des ordres $O_{D_i}$ dans un m\^eme ordre maximal $\Omega=\End(E)$, o\`u $E/\F_q$ est une courbe elliptique supersinguli\`ere. 

Remarquons d'abord que $\phi_1(\alpha_1)$ et $\phi_2(\alpha_2)$ ne commutent pas. En effet, s'ils commutent, $\Q[\phi_1(\alpha_1),\phi_2(\alpha_2)]$ est un sous-corps commutatif de $B_{q\infty}$ strictement plus grand que $\Q$. Donc il est de degr\'e $2$ sur $\Q$ et est isomorphe \`a $\Q[\phi_1(\alpha_1)]$ et \`a $\Q[\phi_2(\alpha_2)]$, ce qui n'est possible que si ces deux derniers corps sont \'egaux. Comme les plongements $\phi_i$ sont optimaux, $\Q[\phi_i(\alpha_i)]\cap\Omega=\Z[\phi_i(\alpha_i)]= \phi_i(O_{D_i})$, ce qui est impossible car, par hypoth\`ese, $O_{D_1}\neq O_{D_2}$.

Comme dans le corollaire \ref{vecteurdegrossshimura}, chacun de ces plongements $\phi_i$ d\'efinit deux plongements optimaux $O_{D_i}\rightarrow \End(E,C_{p,j}^i)$ dans des ordres d'Eichler de niveau $p$, o\`u $i=1,2$, $j=1,2$ et $C_{p,j}^i$ est un sous-groupe d'ordre $p$ de $E$. Plus pr\'ecis\'ement $C_{p,1}^i$ et $C_{p,2}^i$ sont les deux sous-espaces propres du morphisme restreint $\phi_i(\alpha_i)|_{E[p]}$. Pour prouver que pour $p$ assez grand il n'y a pas de plongement optimal de $O_{D_1}$ et $O_{D_2}$ dans un m\^eme ordre d'Eichler de niveau $p$, il suffit de prouver que pour $p$ assez grand $\phi_1(\alpha_1)|_{E[p]}$ et $\phi_2(\alpha_2)|_{E[p]}$ n'ont pas d'espace propre commun. On identifie ces restrictions avec $(\phi_1(\alpha_1) \mod p)$ et $(\phi_2(\alpha_2) \mod p)$ qui sont des \'el\'ements de $\Omega\otimes\F_p\simeq M_2(\F_p)$. Comme $\phi_1(\alpha_1)$ et $\phi_2(\alpha_2)$ ne commutent pas dans $B_{q\infty}$, on a $\Q[\phi_1(\alpha_1),\phi_2(\alpha_2)]=B_{q\infty}$. Donc il existe une constante $M \in\Z$ ne d\'ependant que de $\phi_1(\alpha_1),\phi_2(\alpha_2)$, telle que $\Omega\subseteq \frac{1}{M}\Z[\phi_1(\alpha_1),\phi_2(\alpha_2)]$. Pour $p$ ne divisant pas $M$, on a $\Z[\phi_1(\alpha_1),\phi_2(\alpha_2)]\otimes \F_p=M_2(\F_p)$. Dans ce cas $(\phi_1(\alpha_1) \mod p)$ et $(\phi_2(\alpha_2) \mod p)$ n'ont pas d'espace propre commun, sinon ils engendreraient un sous-groupe de Borel de $M_2(\F_p)$. 

Lorsque $\phi_1$ et $\phi_2$ parcourent les ensembles finis des plongements optimaux de $O_{D_1}$ et $O_{D_2}$ dans les ordres maximaux, $M$ prend un nombre fini de valeurs. Donc pour $p$ sup\'erieur \`a toutes ces valeurs, les vecteurs $\gamma_{D_1}$ et $\gamma_{D_2}$ sont sans ar\^ete commune. 

\end{proof}
\section{Groupe des composantes de $\Jac(X^{pq}/w_q)_{\F_p}$}

On consid\`ere le mod\`ele r\'egulier $\widetilde{{X^{pq}/w_q}}$ sur $\Z_p$ du quotient d'Atkin-Lehner $X^{pq}/w_q$, obtenu par \'eclatement aux points singuliers du mod\`ele d\'ecrit par Cherednik et Drinfeld. Parent et Yafaev ont prouv\'e le pr\'esent lemme sur le groupe des composantes de la fibre en $p$ de la jacobienne du quotient d'Atkin-Lehner $X^{pq}/w_q$. 

\begin{lem}
\label{electro}
Soient $p,q$ deux nombres premiers tels que $g(X_0(q))\ge 5$, $q\equiv 3 \mod 4$, $p \equiv 5 \mod 12$ et $\legendre{q}{p}=-1$. Soient $\mathcal{J}$ la composante irr\'eductible exceptionnelle de $\widetilde{{X^{pq}/w_q}}$ provenant de la multiplication par $\zeta_4$ et $J\neq \mathcal{J}$ une autre composante irr\'eductible de $\widetilde{{X^{pq}/w_q}}$. Pour $p\gg q$, on a $(p+1)(\mathcal{J}-J)\neq 0$ dans le groupe des composantes de $\Jac(X^{pq}/w_q)_{\F_p}$.  
\end{lem}
\begin{proof}
Cela d\'ecoule du lemme \cite[lemma 3.2.3]{ParentYafaev}. 
\end{proof}

Le but de ce paragraphe est de g\'en\'eraliser le lemme \ref{electro} en supprimant la condition $p \equiv -1 \mod 3$. Nous allons prouver le lemme ci-dessous : 
\begin{lem}
\label{general} Soient $p,q$ deux nombres premiers tels que $g(X_0(q))\ge 6$, $q\equiv 3 \mod 4$, $p \equiv 1 \mod 4$ et $\legendre{q}{p}=-1$. Soient $\mathcal{J}$ la composante irr\'eductible exceptionnelle de $\widetilde{{X^{pq}/w_q}}$ provenant de la multiplication par $\zeta_4$ et $J\neq \mathcal{J}$ une autre composante irr\'eductible de $\widetilde{{X^{pq}/w_q}}$. Pour $p\gg q$, on a $(p+1)(\mathcal{J}-J)\neq 0$ dans le groupe des composantes de $\Jac(X^{pq}/w_q)_{\F_p}$. 
\end{lem}

\begin{rmq}
\label{genre}
D'apr\`es \cite[Prop. 1.43]{Shimura}, le genre de $X_0(q)$ est la partie enti\`ere $\lfloor\frac{q+1}{12}\rfloor$ de $\frac{q+1}{12}$ si $q\not\equiv 1 \mod 12$ et $\lfloor\frac{q+1}{12}\rfloor-1$ si $q\equiv 1 \mod 12$. On en d\'eduit que si $q\ge 79$ alors $g(X_0(q))\ge 6$. D'autre part d'apr\`es la d\'efinition \ref{EisensteinModulaire}, le poids du vecteur d'Eisenstein modulaire est $w(A_E)=\sum_{k=1}^{g+1} (\card(\End(E_{j_k})^*/\{\pm 1\}))^{-1},$ o\`u $g=g(X_0(q))$. Comme, d'apr\`es le corollaire \ref{epaisseur}, $\card(\End(E_{j_k})^*/\{\pm 1\})=1$ sauf au plus pour deux valeurs de $j_k$, on a $w(A_E)>g(X_0(q))-1\ge 5$. 
\end{rmq}

\begin{ntt}
\label{notationpar3}
Pour faciliter la compr\'ehension des notations qui suivent, nous invitons le lecteur \`a consulter l'exemple \ref{graphe} ci-dessous. On note $\mathcal{G}(\widetilde{{(X^{pq}/w_q)}}_{\F_p})$ le graphe dual de la fibre en $p$ de $\widetilde{{X^{pq}/w_q}}$. Rappelons que ce graphe est le graphe obtenu \`a partir du graphe $\mathcal{G}((X^{pq}/w_q)_{\F_p})$ en rempla\c cant chaque ar\^ete exceptionnelle de longueur $l$ par une cha\^ine  de $l$ ar\^etes de longueur 1. On renvoie au corollaire \ref{epaisseur2} pour une description de ces ar\^etes. On d\'esigne par $S_1'=S_1/w_q$ et $S_2'=S_2/w_q$ les deux partitions du graphe $\mathcal{G}((X^{pq}/w_q)_{\F_p})$. L'ensemble des sommets de $\mathcal{G}(\widetilde{{(X^{pq}/w_q)}}_{\F_p})$ est l'union disjointe de $S_1'$, $S_2'$ et de l'ensemble des sommets exceptionnels que nous allons d\'ecrire. On note $\mathcal{J}$ le sommet exceptionnel de $\mathcal{G}(\widetilde{{(X^{pq}/w_q)}}_{\F_p})$ provenant de l'\'eclatement de $(X^{pq}/w_q)/ \Z_p$ en l'unique point singulier d'\'epaisseur $2$. Le sommet $\mathcal{J}$ est reli\'e par une ar\^ete \`a chacun des deux sommets $J_1\in S_1'$ et $J_2\in S_2'$ correspondant au $j$-invariant $1728 \mod q$. 
Lorsque $q \equiv -1 \mod 3$ et $p\equiv 1 \mod 3$, on note $G_1\in S_1'$ et $G_2\in S_2'$ les deux sommets correspondant \`a la classe du $j$-invariant $0 \mod q$. Les sommets $G_1$ et $G_2$ sont reli\'es \`a deux sommets exceptionnels $\mathcal{J}_1$ et $\mathcal{J}_2$ provenant de l'unique point singulier d'\'epaisseur $3$ apr\`es \'eclatements successifs. On note par $j_{1,1}, j_{1,2}, \dots, j_{1,l}$ les sommets de $S_1'$ distincts de $J_1$ et $G_1$. De m\^eme $j_{2,1}, j_{2,2}, \dots, j_{2,l}$ d\'esignent les sommets de $S_2'$ distincts de $J_2$ et $G_2$. 

Soient $a$ et $b$ deux sommets du graphe. On note $N(a,b)$ le nombre d'ar\^etes entre les deux sommets $a$ et $b$, et $N(a)$ la puissance du sommet $a$ (\ie le nombre d'ar\^etes qui partent du sommet $a$ ou arrivent au sommet $a$). 
\end{ntt}

\begin{exm}
\label{graphe}
Graphe de $\mathcal{G}(\widetilde{{(X^{13*47}/w_{47})}}_{\F_{13}})$.\\
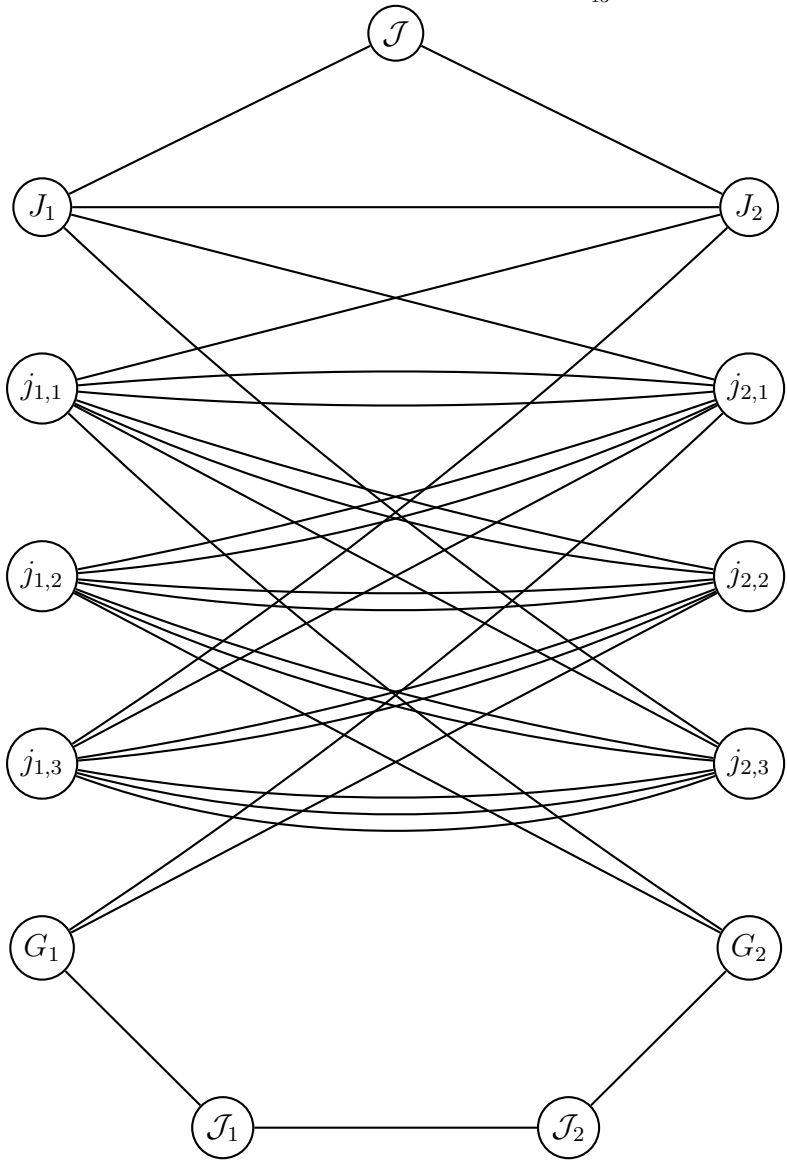
\begin{psmatrix}[mnode=circle]
&&$\mathcal{J}$ \\
$J_1$ & &  & & $J_2$ \\
$j_{1,1}$ &&&& $j_{2,1}$  \\
 $j_{1,2}$ &&&& $j_{2,2}$ \\
$j_{1,3}$ &&&& $j_{2,3}$ \\
$G_1$&  & &  & $G_2$ \\
&$\mathcal{J}_1$ && $\mathcal{J}_2$ \\
\end{psmatrix}

\psset{arrows=-}
\ncline{1,3}{2,1}
\ncline{1,3}{2,5}

\ncline{2,1}{2,5}
\ncline{2,1}{3,5}
\ncline{2,5}{3,1}
\ncarc[arcangle=5]{5,5}{2,1}
\ncarc[arcangle=5]{2,5}{5,1}
\ncarc[arcangle=5]{3,1}{3,5}
\ncarc[arcangle=5]{3,5}{3,1}
\ncarc[arcangle=4]{4,5}{3,1}
\ncarc[arcangle=10]{4,5}{3,1}
\ncarc[arcangle=4]{3,5}{4,1}
\ncarc[arcangle=10]{3,5}{4,1}
\ncline{3,1}{5,5}
\ncline{3,5}{5,1}

\ncarc[arcangle=5]{6,5}{3,1}
\ncarc[arcangle=5]{3,5}{6,1}
\ncarc[arcangle=10]{4,5}{4,1}
\ncarc[arcangle=5]{4,5}{4,1}

\ncarc[arcangle=10]{5,5}{4,1}
\ncarc[arcangle=6]{5,5}{4,1}
\ncarc[arcangle=10]{4,5}{5,1}
\ncarc[arcangle=6]{4,5}{5,1}
\ncline{6,5}{4,1}
\ncline{4,5}{6,1}

\ncarc[arcangle=10]{5,5}{5,1}
\ncarc[arcangle=15]{5,5}{5,1}
\ncarc[arcangle=20]{5,5}{5,1}

\ncline{6,1}{7,2}
\ncline{7,2}{7,4}
\ncline{7,4}{6,5}
\end{exm}
Nous invitons le lecteur \`a se r\'ef\'erer au pr\'esent exemple pour faciliter la compr\'ehension des calculs qui suivent. (On note cependant que ce graphe ne satisfait pas \`a la condition du corollaire \ref{Nnonnul} ci-dessous. En effet dans cet exemple certains des sommets de $S_1'$ ne sont pas reli\'es \`a tous les sommets de $S_2'$).

La preuve du lemme \ref{electro} repose sur la description du groupe des composantes donn\'ee par Raynaud \cite[thm 1, p 274]{BLR}.   
\begin{lem}
\label{loiK}
(\og Loi K \fg). Avec les notations du lemme \ref{general}, on a $(p+1)(\mathcal{J}-J)= 0$ dans le groupe des composantes de $\Jac(X^{pq}/w_q)_{\F_p}$, si et seulement s'il existe une fonction $\nu$ de l'ensemble des sommets $S$ de $\mathcal{G}(\widetilde{{(X^{pq}/w_q)}}_{\F_p})$ dans $\Z$ telle que pour chaque sommet $C$ de $\mathcal{G}(\widetilde{{(X^{pq}/w_q)}}_{\F_p})$, on a l'\'egalit\'e : 
$$\sum_{D\mapsto C}  \nu(C)-\nu(D)=\begin{cases} p+1 \text{ si } C=\mathcal{J}; \\ -(p+1) \text{ si } C=J; \\ 0 \text{ sinon.} \end{cases}$$
O\`u $\sum_{D\mapsto C}$ d\'esigne la somme faite sur tous les sommets $D$ voisins du sommet $C$ avec une multiplicit\'e \'egale au nombre d'ar\^etes entre $D$ et $C$. C'est-\`a-dire : 
$$\sum_{D\in S}  N(C,D)(\nu(C)-\nu(D))=\begin{cases} p+1 \text{ si } C=\mathcal{J}; \\ -(p+1) \text{ si } C=J; \\ 0 \text{ sinon.} \end{cases}$$
\end{lem}
\begin{proof}
Consulter \cite[Sublemma 3.1.3.1]{ParentYafaev}. 
\end{proof}
Le lemme \ref{loiK} peut se traduire de mani\`ere \'electrodynamique, on pourra consulter \`a ce sujet \cite[\S 3.1]{ParentYafaev} ou \cite[II.1]{Edixhoven}. Identifions le graphe $\mathcal{G}(\widetilde{{(X^{pq}/w_q)}}_{\F_p})$ \`a un circuit \'electrique de la mani\`ere suivante : les sommets du graphe sont les n\oe uds du circuit et les ar\^etes du graphe sont des fils reliant ces n\oe uds. Chacun de ces fils \'etant d'une r\'esistance de $1$ Ohm. Supposons que les n\oe uds $\mathcal{J}$ et $J$ soient reli\'ees aux p\^oles d'un g\'en\'erateur \'electrique, et qu'un courant de $p+1$ Amp\`eres entre dans le circuit en le n\oe ud $J$ et quitte le circuit au n\oe ud $\mathcal{J}$. Soit $\nu$ une r\'epartition des potentiels, exprim\'ee en Volts, sur l'ensemble des n\oe uds du circuit, d\'efinie \`a une r\'epartition constante pr\`es. La diff\'erence de potentiel entre deux n\oe uds reli\'es par un fil $A$ et $B$ est $u(A,B)=\nu(A)-\nu(B)$. Chaque fil entre $A$ et $B$ est parcouru par un courant (\'eventuellement d'intensit\'e nulle) dirig\'e de $A$ vers $B$ si $u(A,B)\ge 0$ et de $B$ vers $A$ si $u(A,B)\le 0$. La loi d'Ohm affirme que si le fil $c$, de r\'esistance $r(c)$ est parcouru par un courant dirig\'e de $A$ vers $B$, dont l'intensit\'e est $i(c)$ Amp\`eres, alors la diff\'erence de potentiel $u(A,B)$ est \'egale \`a $r(c)i(c)$ Volts. Dans le cas qui nous int\'eresse $i(c)=u(A,B)=\nu(A)-\nu(B)$ comme $r(c)=1$. La loi des n\oe uds de Kirchhoff affirme que pour chaque n\oe ud $A$, la somme des intensit\'e des courants arrivant au n\oe ud $A$ est \'egale \`a la somme des intensit\'es des courants partant du sommet $A$. En combinant la loi de Kirchhoff en un n\oe ud $C$ avec la loi d'Ohm sur tous les fils partant de $C$, on retrouve les \'equations du lemme \ref{loiK}. Or d'apr\`es \cite[thm 1, p 274]{BLR} la solution $\nu$ de ce syst\`eme est unique \`a une fonction constante pr\`es. Donc prouver que $(p+1)(\mathcal{J}-J)\neq 0$ revient \`a prouver que le circuit \'electrique d\'ecrit ne peut pas admettre de r\'epartition de potentiel $\nu$ \`a valeurs enti\`eres.

Les lemmes suivants d\'ecrivent la r\'epartition des ar\^etes de $\mathcal{G}(\widetilde{{(X^{pq}/w_q)}}_{\F_p})$. 
\begin{lem}
\label{repartition}
Avec les hypoth\`eses du lemme \ref{general}, fixons $j\in\F_{q^2}$ un $j$-invariant supersingulier d\'efini \`a action de $\Gal(\F_{q^2}/\F_q)$ pr\`es, et soit $C$ le sommet de $S_1'$ ou $S_2'$ associ\'e \`a ce $j$-invariant. 
Le nombre $N(C)$ d'ar\^etes qui partent de $C$ dans le graphe $\mathcal{G}(\widetilde{{(X^{pq}/w_q)}}_{\F_p})$ est donn\'e par : 
\begin{itemize}
\item[] $(p+1)$ si $j$ appartient \`a $\F_{q^2}\setminus \F_q$; 
\item[] $(p+1)/2$ si $j$ appartient \`a $\F_q$ et $j\not\equiv 0,\ 1728$; 
\item[] $(p+3)/4$ si $j\equiv 1728$; 
\item[] $(p+1)/6$ si $j\equiv 0$ et $p\equiv -1 \mod 3$; 
\item[] $(p+5)/6$ si $j\equiv 0$ et $p\equiv 1 \mod 3$.
\end{itemize}
\end{lem}
\begin{proof}
Le nombre d'ar\^etes partant de chaque sommet non exceptionnel du gra\-phe $\mathcal{G}(\widetilde{{(X^{pq}/w_q)}}_{\F_p})$ est \'egal au nombre d'ar\^etes partant du sommet correspondant dans le graphe $\mathcal{G}({(X^{pq}/w_q)}_{\F_p})$. Le lemme \cite[lemma 3.1.2]{ParentYafaev} et sa preuve donnent le r\'esultat. 
\end{proof}

\begin{lem}
\label{repar2}
Soient $j_1,j_2\in\F_{q^2}$ deux $j$-invariants supersinguliers d\'efinis \`a action de $\Gal(\F_{q^2}/\F_q)$ pr\`es. On note $C_1\in S_1'$ et $C_2\in S_2'$ les sommets associ\'es \`a ces $j$-invariants. On note $\epsilon(C_1,C_2)=2$ si $j_1\in\F_q$ et $j_2\in\F_q$ et $\epsilon(C_1,C_2)=1$ si $j_1\in\F_{q^2}\setminus\F_q$ ou $j_2\in\F_{q^2}\setminus\F_q$. Le nombre d'ar\^etes $N(C_1,C_2)$ entre les sommets $C_1$ et $C_2$ dans le graphe $\mathcal{G}(\widetilde{{(X^{pq}/w_q)}}_{\F_p})$ est : 
$$ \frac{p+1}{w(A_E)}\times\frac{1}{\epsilon(C_1,C_2)w(C_1)w(C_2)}+O_q(\sqrt{p}). $$
O\`u $w(A_E)=\sum_{k=1}^{g+1} (\card(\End(E_{j_k})^*/\{\pm 1\}))^{-1}$ est le degr\'e du diviseur $A_E\in\mathcal{P}_S$, et $w(C_i)=\card(\End(E_{j_i})^*/\{\pm 1\})$. 
\end{lem}

\begin{proof}
Il suffit d'appliquer le lemme \cite[lemma 3.1.2]{ParentYafaev} qui donne le nombre d'ar\^etes $N$ entre les deux sommets $C_1$ et $C_2$ dans $\mathcal{G}({(X^{pq}/w_q)}_{\F_p})$. Le nombre d'ar\^etes $N(C_1,C_2)$ liant les deux sommets non exceptionnels $C_1$ et $C_2$ dans $\mathcal{G}(\widetilde{{(X^{pq}/w_q)}}_{\F_p})$ \'etant \'egal \`a $N$, ou \`a $N-1$ si $\{C_1,\ C_2\}=\{J_1,\ J_2 \}$ ou $\{C_1,\ C_2\}=\{G_1, \ G_2\}$ et $p\equiv 1 \mod 3$.
\end{proof}

Une cons\'equence utile du lemme \ref{repar2} est le corollaire suivant. 
\begin{cor}
\label{Nnonnul}
Pour $p\gg q$, on a $N(C_1,C_2)> 0$ pour tout $C_1\in S_1'$ et $C_2\in S_2'$. 
\end{cor}
Dans la suite nous supposerons que cette condition est satisfaite.

Pour prouver le lemme \ref{general}, nous devons traiter le cas o\`u $p \equiv 1 \mod 3$. 
Si $q\equiv 1 \mod 3$, la seule composante exceptionnelle est $\mathcal{J}$ et la preuve de \cite[lemma 3.1.3]{ParentYafaev} reste valable. On se place dans le cas o\`u $q\equiv -1 \mod 3$. Supposons d'abord que $J$ n'est pas une des composantes exceptionnelles $\mathcal{J}_1$ ou $\mathcal{J}_2$. 
Soient $i$ et $j$ tels que $\{i,j\}=\{1,2\}$. On note $I(G_i)=0$ si $J\neq G_i$ et $I(G_i)=p+1$ si $J=G_i$. 
En appliquant la loi K (lemme \ref{loiK}) successivement aux sommets $\mathcal{J}_i$ et $\mathcal{J}_j$, on trouve : 
$$(\nu(\mathcal{J}_i)-\nu(\mathcal{J}_j))+ (\nu(\mathcal{J}_i)-\nu(G_i))=0,$$
$$(\nu(\mathcal{J}_j)-\nu(\mathcal{J}_i))+ (\nu(\mathcal{J}_j)-\nu(G_j))=0.$$
Soit : 
$$ \nu(G_i)-\nu(\mathcal{J}_i)= \nu(\mathcal{J}_i)-\nu(\mathcal{J}_j)=\nu(\mathcal{J}_j)-\nu(G_j).$$
Donc : 
\begin{equation}
\label{3resseries}
 \nu(G_i)-\nu(G_j)=3(\nu(G_i)-\nu(\mathcal{J}_i)). 
\end{equation}
D'autre part en appliquant la loi K au sommet $G_i$, on trouve : 
$$ N(G_i,G_j)(\nu(G_i)-\nu(G_j)) + (\nu(G_i)-\nu(\mathcal{J}_i))+ \sum_{D\in S_j'\setminus\{G_j\}}N(G_i,D)(\nu(G_i)-\nu(D))=I(G_i).$$
En combinant avec l'\'egalit\'e (\ref{3resseries}) (ou directement en appliquant la loi des r\'esistances en s\'erie), on trouve :  
$$(N(G_i,G_j)+\frac{1}{3})(\nu(G_i)-\nu(G_j))+\sum_{D\in S_j'\setminus\{G_j\}}N(G_i,D)(\nu(G_i)-\nu(D))=I(G_i). $$
Pour $C_1$ et $C_2$ diff\'erents de $\mathcal{J}_1$ et de $\mathcal{J}_2$ on pose : 
$$N'(C_1,C_2)=\begin{cases}N(C_1,C_2)+1/3 \text{ si }\{C_1,C_2\}=\{G_1,G_2\};\\ N(C_1,C_2)\text{ sinon.} \end{cases}$$
Ceci nous permet de consid\'erer un syst\`eme d'\'equations dont les inconnues sont les $\nu(C)$ o\`u $C\in S_1'\sqcup S_2' \sqcup \{ \mathcal{J}\}$, c'est-\`a-dire $C\neq\mathcal{J}_1,\mathcal{J}_2$, et dont les coefficients sont les $N'(C_1,C_2)$ pour $C_1,C_2\in S_1'\sqcup S_2' \sqcup \{ \mathcal{J}\}$. Pour prouver que ce syst\`eme est sans solution enti\`ere, reprenons la preuve de \cite[lemma 3.1.3]{ParentYafaev} en rempla\c cant dans toutes les \'equations $N(C_1,C_2)$ par $N'(C_1,C_2)$. La preuve de \cite[lemma 3.1.3]{ParentYafaev} utilise les valeurs asymptotiques (\`a $O_q(\sqrt{p})$ pr\`es) de $N(C_1,C_2)$. Ces valeurs \'etant conserv\'ees lorsque on remplace $N(C_1,C_2)$ par $N'(C_1,C_2)$. On obtient une preuve du lemme \ref{general} dans le cas o\`u $J\neq\mathcal{J}_1,\mathcal{J}_2$. 

 Il reste \`a prouver le lemme \ref{general} dans le cas o\`u $J=\mathcal{J}_1$ ou $\mathcal{J}_2$. Comme l'involution $w_p$ \'echange $\mathcal{J}_1$ et $\mathcal{J}_2$ et laisse invariant $\mathcal{J}$, il suffit de prouver le lemme suivant :

\begin{lem}
\label{J-J1}
Avec les notations du lemme \ref{general}, supposons que $q \equiv -1 \mod 3$ et $p\equiv 1 \mod 3$. Pour $p\gg q$, on a $(p+1)(\mathcal{J}-\mathcal{J}_1)\neq 0$ dans le groupe des composantes de $\Jac(X^{pq}/w_q)_{\F_p}$. 
\end{lem}

\begin{proof}
On suppose qu'il existe une r\'epartition de potentiels entiers $\nu$, sur les sommets de $\mathcal{G}(\widetilde{{(X^{pq}/w_q)}}_{\F_p})$, tel qu'un courant de $p+1$ Amp\`eres entre dans le circuit au sommet $\mathcal{J}_1$ et le quitte en $\mathcal{J}$. 

Nous allons prouver que pour $p\gg q$ les potentiels $\nu(j_{1,1})$, $\nu(j_{1,2}),\dots,\nu(j_{1,l})$ et $\nu(j_{2,1})$,\ $\nu(j_{2,2}), \dots,  \nu(j_{2,l})$ sont tous \'egaux (cf. sous-lemme \ref{egalite}). Puis nous obtiendrons une contradiction (cf. sous-lemmes \ref{congruence} et \ref{inferiorite}). Fixons $J_M$ un sommet ayant un potentiel maximal dans $S_1'\cup S_2' \setminus \{G_1, G_2 \}$. Ainsi que $J_m$ un sommet ayant un potentiel minimal dans $S_1'\cup S_2' \setminus \{J_1, J_2 \}$. 

\begin{sublem}
\label{JMax1}
Soient $i,j$ tels que $\{i,j\}=\{1,2\}$ et $J_M\in S_j'$. On a la majoration suivante pour l'intensit\'e du courant que $J_{M}$ re\c coit de $G_i$ :  
$$N(J_{M},G_i)(\nu(G_i)-\nu(J_{M}))\le \frac{6(p+1)}{(3w(A_E)-1)\epsilon(G_i,J_{M})w(J_{M})}+O_q(\sqrt{p}).  $$
\end{sublem}
\begin{proof}
Remarquons que si $\nu(G_i)\le \nu(J_M)$ la proposition est \'evidente. On va donc supposer $\nu(G_i)>\nu(J_M)$. Par la loi K exprim\'ee en le sommet $G_i$, 
$$\sum_{D\mapsto G_i} (\nu(G_i)-\nu(D))=0,$$
c'est-\`a-dire
\begin{equation}
\label{loiKGi}
\sum_{\topbot{D\mapsto G_i}{D\neq G_j,\mathcal{J}_i }} (\nu(G_i)-\nu(D))=N(G_j,G_i)(\nu(G_j)-\nu(G_i))+(\nu(\mathcal{J}_i)-\nu(G_i)).
\end{equation}
Nous allons majorer $N(G_j,G_i)(\nu(G_j)-\nu(G_i))+(\nu(\mathcal{J}_i)-\nu(G_i)).$ 
En appliquant la loi K au sommet $\mathcal{J}_1$ on trouve : 
$$(\nu(\mathcal{J}_1)-\nu(G_1))+(\nu(\mathcal{J}_1)-\nu(\mathcal{J}_2))=p+1;$$
et en appliquant la loi K au sommet $\mathcal{J}_2$ on obtient
\begin{equation}
\label{loiKJ2}
\nu(\mathcal{J}_1)-\nu(\mathcal{J}_2)=\nu(\mathcal{J}_2)-\nu(G_2). 
\end{equation}
Donc on a 
\begin{equation}
\label{J1G1J2G2}
(\nu(\mathcal{J}_1)-\nu(G_1))+(\nu(\mathcal{J}_2)-\nu(G_2))=p+1.
\end{equation}
D'apr\`es \cite[Sublemma 3.1.3.2]{ParentYafaev}, on a $\nu(\mathcal{J}_1)\ge \nu(G_1),\nu(\mathcal{J}_2)$ donc $\nu(\mathcal{J}_1)-\nu(G_1)\ge 0$ et $\nu(\mathcal{J}_2)-\nu(G_2)=\nu(\mathcal{J}_1)-\nu(\mathcal{J}_2)\ge 0$ d'apr\`es l'\'egalit\'e (\ref{loiKJ2}). On a donc
\begin{equation}
\label{demimino1}
\nu(\mathcal{J}_i)-\nu(G_i)\le p+1,
\end{equation}
On applique la loi K au sommet $G_j$ : 
$$\sum_{\topbot{D\mapsto G_j}{D\neq G_i,\mathcal{J}_j }} (\nu(D)-\nu(G_j))=N(G_j,G_i)(\nu(G_j)-\nu(G_i))+(\nu(G_j)-\nu(\mathcal{J}_j)).$$
Si on suppose $\nu(G_j)>\nu(G_i)$, on a par hypoth\`ese $ \nu(G_i)>\nu(J_M) \ge \nu(D)$ pour tout $D\in (S_1'\sqcup S_2')\setminus \{ G_1,G_2\}$, on en d\'eduit $\nu(G_j)>\nu(D)$ et 
\begin{equation}
\label{demimino2}
N(G_i,G_j)(\nu(G_j)-\nu(G_i)) \le \nu(\mathcal{J}_j)-\nu(G_j).
\end{equation}
En combinant l'\'egalit\'e (\ref{J1G1J2G2}) avec l'in\'egalit\'e (\ref{demimino2}) si $\nu(G_j)>\nu(G_i)$ et en consid\'erant l'in\'ega\-lit\'e (\ref{demimino1}) si $\nu(G_j)\le\nu(G_i)$, on trouve l'in\'egalit\'e suivante : 
\begin{equation}
\label{minop+1}
N(G_j,G_i)(\nu(G_j)-\nu(G_i))+(\nu(\mathcal{J}_i)-\nu(G_i))\le p+1.
\end{equation}
En combinant l'\'egalit\'e (\ref{loiKGi}) avec la majoration (\ref{minop+1}), on trouve : 
$$\sum_{\topbot{D\mapsto G_i}{D\neq G_j,\mathcal{J}_i }} (\nu(G_i)-\nu(D))\le p+1.$$
Comme $J_{M}$ a un potentiel maximal parmi tous les sommets adjacents \`a $G_i$ diff\'erents de $G_j$ et $\mathcal{J}_i$, on a
$$\sum_{\topbot{D\mapsto G_i}{D\neq G_j,\mathcal{J}_i }} (\nu(G_i)-\nu(J_{M}))\le p+1.$$
Soit : 
$$\big(N(G_i)-N(G_i,G_j)-N(G_i,\mathcal{J}_i)\big)(\nu(G_i)-\nu(J_{M}))\le p+1 $$
Soit $C$ l'intensit\'e du courant que $J_{M}$ re\c coit de $G_i$, on a l'in\'egalit\'e : 

$$C=N(G_i,J_{M})(\nu(G_i)-\nu(J_{M}))\le N(G_i,J_{M})\frac{(p+1)}{(N(G_i)-N(G_i,G_j)-1)}. $$

On utilise les lemmes \ref{repartition} et \ref{repar2} : 
$$ C \le  \frac{p+1}{(p+5)/6-1-(p+1)/(18w(A_E))+O_q(\sqrt{p})}\bigg(\frac{p+1}{w(A_E)\epsilon(G_i,J_{M})3w(J_{M})} +O_q(\sqrt{p}) \bigg).$$
$$\le \bigg(\frac{18w(A_E)}{3w(A_E)-1}+\frac{O_q(\sqrt{p})}{p+1}\bigg)\bigg( \frac{p+1}{w(A_E)\epsilon(G_i,J_{M})3w(J_{M})}+O_q(\sqrt{p})  \bigg). $$
Finalement : 
$$C\le \frac{6(p+1)}{(3w(A_E)-1)\epsilon(G_i,J_{M})w(J_{M})}+O_q(\sqrt{p}). $$

\end{proof}

On d\'emontre un r\'esultat similaire pour $J_m$. 

\begin{sublem}
\label{JMin1}
Soient $i,j$ tels que $\{i,j\}=\{1,2\}$ et $J_m\in S_j'$.
On a la majoration suivante pour le courant que $J_{m}$ donne \`a $J_i$ : 
$$N(J_{m},J_i)(\nu(J_{m})-\nu(J_i))\le \frac{4(p+1)}{(2w(A_E)-1)\epsilon(J_i,J_{m})w(J_{m})} + O_q(\sqrt{p}). $$
\end{sublem}
\begin{proof}
Remarquons que si $\nu(J_i)\ge \nu(J_m)$ la proposition est \'evidente, nous supposerons dans la suite que $\nu(J_i)< \nu(J_m)$. Par la loi K exprim\'ee en le sommet $J_i$, 

\begin{equation}
\label{loiKJi}
\sum_{\topbot{D\mapsto J_i}{D\neq J_j,\mathcal{J} }} \nu(D)-\nu(J_i)=N(J_i,J_j)(\nu(J_i)-\nu(J_j))+(\nu(J_i)-\nu(\mathcal{J})).
\end{equation}
Nous allons majorer la quantit\'e $N(J_i,J_j)(\nu(J_i)-\nu(J_j))+(\nu(J_i)-\nu(\mathcal{J}))$. 
En appliquant la loi K au sommet $\mathcal{J}$, on trouve : 
\begin{equation}
\label{loiKJ2bis}
(\nu(J_i)-\nu(\mathcal{J}))+(\nu(J_j)-\nu(\mathcal{J}))=p+1
\end{equation}
D'apr\`es \cite[Sublemma 3.1.3.2]{ParentYafaev} on a $\nu(\mathcal{J})\le \nu(J_j)$, donc 
\begin{equation}
\label{2demimino1}
\nu(J_i)-\nu(\mathcal{J})\le p+1
\end{equation}
On applique la loi K au sommet $J_j$ : 
$$\sum_{\topbot{D\mapsto J_j}{D\neq J_i,\mathcal{J} }} \nu(D)-\nu(J_j)=N(J_i,J_j)(\nu(J_j)-\nu(J_i))+\nu(J_j)-\nu(\mathcal{J}).$$
Si $\nu(J_j)< \nu(J_i)$, alors comme par hypoth\`ese $\nu(J_i)<\nu(J_m)\le \nu(D)$ pour tout sommet $D\in S_1'\sqcup S_2' \setminus\{ J_1,J_2\}$, on a $\nu(J_j)\le \nu(D)$ et
\begin{equation}
\label{2demimino2}
N(J_i,J_j)(\nu(J_i)-\nu(J_j))\le \nu(J_j)-\nu(\mathcal{J})
\end{equation}
En combinant l'in\'egalit\'e (\ref{2demimino2}) avec l'\'egalit\'e (\ref{loiKJ2bis}) si $\nu(J_j)< \nu(J_i)$ et en consid\'erant l'in\'egalit\'e (\ref{2demimino1}) si $\nu(J_j)\ge \nu(J_i)$, on trouve la majoration : 
\begin{equation}
\label{2minop+1}
N(J_i,J_j)(\nu(J_i)-\nu(J_j))+(\nu(J_i)-\nu(\mathcal{J})) \le p+1
\end{equation}
En combinant cette in\'egalit\'e (\ref{2minop+1}) avec l'\'egalit\'e (\ref{loiKJi}), on trouve : 
$$\sum_{\topbot{D\mapsto J_i}{D\neq J_j,\mathcal{J} }} \nu(D)-\nu(J_i)\le p+1.$$
Comme $J_{m}$ a un potentiel minimal parmi tous les sommets adjacents \`a $J_i$ diff\'erents de $J_j$ et $\mathcal{J}$, on a 
$$(N(J_i)-N(J_i,\mathcal{J})-N(J_i,J_j))(\nu(J_{m})-\nu(J_i))\le p+1, $$
donc
$$\nu(J_{m})-\nu(J_i)\le \frac{p+1}{N(J_i)-1-N(J_i,J_j)}. $$
On utilise les lemmes \ref{repartition} et \ref{repar2} :
$$\nu(J_{m})-\nu(J_i)\le \frac{p+1}{(p+3)/4-1-(p+1)/(w(A_E)8)+O_q(\sqrt{p})}$$
Soit 
$$\nu(J_{m})-\nu(J_i)\le \frac{8w(A_E)}{2w(A_E)-1}+\frac{O_q(\sqrt{p})}{p+1}.  $$
D'autre part, le courant $C$ que $J_{m}$ donne \`a $J_i$ est : 
$$C=N(J_i,J_{m})(\nu(J_{m})-\nu(J_i)).$$ 
$$=\bigg( \frac{p+1}{w(A_E)2\epsilon(J_i,J_{m})w(J_{m})}+O_q(\sqrt{p})\bigg) (\nu(J_{m})-\nu(J_i)). $$
D'o\`u l'in\'egalit\'e :  
$$ C \le \frac{4(p+1)}{(2w(A_E)-1)\epsilon(J_i,J_{m})w(J_{m})} + O_q(\sqrt{p}).$$
\end{proof}

Nous allons prouver que $J_M$ ne peut pas \^etre \'egal \`a $J_1$ ou $J_2$ et que $J_m$ ne peut pas \^etre $G_1$ ou $G_2$. 
\begin{sublem}
\label{decr}
Pour $p\gg q$, on a les in\'egalit\'es : $\nu(J_1),\nu(J_2)<\nu(J_M)$ et $\nu(G_1)$, $\nu(G_2)>\nu(J_m)$. 
\end{sublem}
\begin{proof}
Supposons que pour $j=1$ ou $2$ on a $J_j=J_M$, nous allons obtenir une contradiction pour $p\gg q$. 
En appliquant la loi de Kirchhoff au sommet $J_j$ on obtient : 
$$\sum_{\topbot{D\mapsto J_j}{D\neq \mathcal{J},G_i}} \nu(D)-\nu(J_j)=N(J_j,G_i)(\nu(J_j)-\nu(G_i))+\nu(J_j)-\nu(\mathcal{J}); $$
o\`u $i$ est tel que $\{i,j\}=\{1,2\}$.
Par hypoth\`ese, pour tout sommet $D$ adjacent \`a $J_j$ et diff\'erent de $\mathcal{J}$ et de $G_i$, on a $\nu(D)\le \nu(J_j)$. D'o\`u : 
\begin{equation}
\label{E11}
N(J_j,G_i)(\nu(G_i)-\nu(J_j))\ge \nu(J_j)-\nu(\mathcal{J}).
\end{equation}
On applique le sous-lemme \ref{JMax1} \`a $J_j=J_M$ : 
\begin{equation}
\label{E12}
N(J_{j},G_i)(\nu(G_i)-\nu(J_{j}))\le \frac{3(p+1)}{2(3w(A_E)-1)}+O_q(\sqrt{p}).
\end{equation}
D'autre part en appliquant la loi K \`a $\mathcal{J}$ on trouve : 
$$\nu(\mathcal{J})-\nu(J_1)+\nu(\mathcal{J})-\nu(J_2)=-(p+1),$$
comme $\nu(J_1),\nu(J_2)\le \nu(J_j)$ on en d\'eduit que : 
\begin{equation}
\label{E13}
\nu(J_j)-\nu(\mathcal{J})\ge (p+1)/2.
\end{equation}

En combinant ces trois in\'egalit\'es (\ref{E11}), (\ref{E12}) et (\ref{E13}) on trouve : 
$$\frac{p+1}{2} \le \frac{3(p+1)}{2(3w(A_E)-1)}+O_q(\sqrt{p}).$$
Ce qui est contradictoire pour $p\gg q$ comme $w(A_E)>4/3$ d'apr\`es la remarque \ref{genre}. 

Supposons que pour $j=1$ ou $2$ on a $G_j=J_m$, nous allons obtenir une contradiction pour $p\gg q$. 
En appliquant la loi de Kirchhoff au sommet $G_j$ on obtient : 
$$\sum_{\topbot{D\mapsto G_j}{D\neq \mathcal{J}_j,J_i}} \nu(D)-\nu(G_j)=N(G_j,J_i)(\nu(G_j)-\nu(J_i))+\nu(G_j)-\nu(\mathcal{J}_j); $$
o\`u $i$ est tel que $\{i,j\}=\{1,2\}$.
Par hypoth\`ese, pour tout sommet $D$ adjacent \`a $G_j$ et diff\'erent de $\mathcal{J}_j$ et de $J_i$, on a $\nu(D)\ge \nu(G_j)$. D'o\`u :
$$\nu(\mathcal{J}_j)-\nu(G_j) \le  N(G_j,J_i)(\nu(G_j)-\nu(J_i)),$$
on applique le sous-lemme \ref{JMin1} \`a $G_j=J_m$ : 
\begin{equation}
\label{J1}
\nu(\mathcal{J}_j)-\nu(G_j)\le \frac{2(p+1)}{3(2w(A_E)-1)} + O_q(\sqrt{p}).
\end{equation}
D'autre part, en appliquant la loi K au sommet $\mathcal{J}_1$ 
\begin{equation}
\label{LoiKmathcalJ1}
\nu(\mathcal{J}_1)-\nu(G_1)+\nu(\mathcal{J}_1)-\nu(\mathcal{J}_2)=p+1, 
\end{equation}
et en appliquant la loi K au sommet $\mathcal{J}_2$ 
$$\nu(\mathcal{J}_2)-\nu(G_2)+\nu(\mathcal{J}_2)-\nu(\mathcal{J}_1)=0, $$
soit 
\begin{equation}
\label{LoiKmathcalJ2}
\nu(\mathcal{J}_1)-\nu(G_2)=2(\nu(\mathcal{J}_1)-\nu(\mathcal{J}_2). 
\end{equation}
En combinant les deux \'egalit\'es (\ref{LoiKmathcalJ1}) et (\ref{LoiKmathcalJ2}), on obtient : 
\begin{equation}
\label{resistancesenseries}
\nu(\mathcal{J}_1)-\nu(G_1)+\frac{1}{2}(\nu(\mathcal{J}_1)-\nu(G_2))=p+1. 
\end{equation}
Comme $\nu(G_j) \le \nu(G_1),\nu(G_2)$, on a l'in\'egalit\'e : 
$3/2(\nu(\mathcal{J}_1)-\nu(G_j))\ge p+1.$
Si $j=1$ on en d\'eduit que $\nu(\mathcal{J}_1)-\nu(G_1) \ge 2/3(p+1)$. Si $j=2$, en utilisant (\ref{LoiKmathcalJ2}), on obtient $\nu(\mathcal{J}_2)-\nu(G_2) \ge 1/3(p+1)$. Dans tous les cas on a : 
$$\nu(\mathcal{J}_j)-\nu(G_j)\ge \frac{1}{3}(p+1).$$
En combinant avec l'in\'egalit\'e (\ref{J1}) : 
$$\frac{1}{3}(p+1) \le \frac{2(p+1)}{3(2w(A_E)-1)} + O_q(\sqrt{p}).$$
Ce qui est impossible si $p\gg q$ comme $w(A_E)>3/2$ d'apr\`es la remarque \ref{genre}. 
\end{proof}
On suppose dans la suite que les hypoth\`eses du sous-lemme \ref{decr} ci-dessus sont satisfaites.

\begin{sublem}
\label{JMcontr}
Soient $i,j$ tels que $\{i,j\}=\{1,2\}$ et $J_M\in S_j'$. Pour $p\gg q$, le potentiel $\nu(J_{M})$ ne peut pas \^etre strictement sup\'erieur au potentiel de 3 sommets ou plus de $S_i'$. 
\end{sublem}
\begin{proof}
D'apr\`es les hypoth\`eses, $\nu(J_{M})$ est sup\'erieur aux potentiels de tous les sommets de $S_i'\setminus \{G_i\}$. Supposons que $\nu(J_{M})$ est strictement sup\'erieur au potentiel de 3 sommets $D_1,D_2, D_3$ de $S_i'$. On applique la loi K au sommet $J_{M}$ en remarquant que $J_M\neq J_j$ d'apr\`es le sous-lemme \ref{decr} : 
$$ \sum_{\topbot{D\mapsto J_{M}}{D\neq G_i}} \nu(J_{M})-\nu(D)=N(J_{M},G_i)(\nu(G_i)-\nu(J_{M})).$$
Comme $\nu(J_{M})-\nu(D)\ge 0$ pour $D\in S_i'\setminus \{G_i\}$, on a $\nu(G_i)-\nu(J_M)\ge 0$, donc $G_i\not\in\{D_1,D_2,D_3\}.$ On a donc
\begin{equation}
\label{sum3max}
\sum_{k=1}^3 N(J_{M},D_k) (\nu(J_{M})-\nu(D_k))\le N(J_{M},G_i)(\nu(G_i)-\nu(J_{M})). 
\end{equation}
Par hypoth\`ese, $\nu(J_{M})-\nu(D_k)\ge 1$, on en d\'eduit l'in\'egalit\'e : 
$$\sum_{k=1}^3 N(J_{M},D_k)\le N(J_{M},G_i)(\nu(G_i)-\nu(J_{M})).$$
On rappelle que $$N(J_{M},D_k)=\frac{p+1}{w(A_E)\epsilon(J_{M},D_k)w(J_{M})w(D_k)}+O_q(\sqrt{p}).$$
Comme $\epsilon(J_{M},D_k)\le \epsilon(G_i,J_{M})$ et $w(D_k)=1$ sauf pour $D_k=J_i$, auquel cas $w(J_i)=2$. 
$$\sum_{k=1}^3 N(J_{M},D_k)\ge \frac{5}{2}\times\frac{(p+1)}{w(A_E)\epsilon(G_i,J_{M})w(J_{M})}+O_q(\sqrt{p}).$$
En combinant avec l'in\'egalit\'e (\ref{sum3max}) et le sous-lemme \ref{JMax1} on obtient : 
$$\frac{5(p+1)}{2w(A_E)\epsilon(G_i,J_{M})w(J_{M})}+O_q(\sqrt{p}) \le \frac{6(p+1)}{(3w(A_E)-1)w(J_{M})\epsilon(G_i,J_{M})}+O_q(\sqrt{p}).$$
Comme d'apr\`es la remarque \ref{genre} $w(A_E)>5/2$, on obtient une contradiction pour $p\gg q$. 
\end{proof}
On d\'emontre un r\'esultat similaire pour $J_m$. 
\begin{sublem}
\label{Jmcontr}
Soient $i,j$ tels que $\{i,j\}=\{1,2\}$ et $J_m\in S_j'$.
Le potentiel $\nu(J_{m})$ ne peut pas \^etre strictement inf\'erieur au potentiel de 3 sommets ou plus de $S_i'$. 
\end{sublem}
\begin{proof}
D'apr\`es les hypoth\`eses, $\nu(J_{m})$ est inf\'erieur \`a tous les potentiels des sommets de $S_i'\setminus \{ J_i\}$. Supposons en outre que $\nu(J_{m})$ est strictement inf\'erieur au potentiel de 3 sommets $D_1,D_2,D_3$ de $S_i'$. On applique la loi K au sommet $J_{m}$ en notant que $J_m\neq J_j$ par le sous-lemme \ref{decr} : 
$$ \sum_{\topbot{D\mapsto J_{m}}{D\neq J_i}} \nu(D)-\nu(J_{m})=N(J_{m},J_i)(\nu(J_{m})-\nu(J_i)).$$
Comme $\nu(D)-\nu(J_{m})\ge 0$ pour tout $D\in S_2'\setminus \{J_i\}$, on a $\nu(J_{m})-\nu(J_i)\ge 0$, donc $J_i\not\in\{D_1,D_2,D_3\}$. 
 On a : 
$$ \sum_{k=1}^3 N(D_k,J_{m})(\nu(D_k)-\nu(J_{m}))\le N(J_{m},J_i)(\nu(J_{m})-\nu(J_i)). $$
Comme par hypoth\`ese, $\nu(D_k)-\nu(J_{m})\ge 1$ pour $i=1,2,3$, on a : 
$$ \sum_{k=1}^3 N(D_k,J_{m})\le N(J_{m},J_i)(\nu(J_{m})-\nu(J_i)). $$

Soit en appliquant le lemme \ref{repar2}
$$\sum_{k=1}^3 \frac{p+1}{w(A_E)\epsilon(J_{m},D_k)w(J_{m})w(D_k)}+O_q(\sqrt{p})\le N(J_{m},J_i)(\nu(J_{m})-\nu(J_i)).$$
Remarquons que $\epsilon(J_{m},D_k) \le \epsilon(J_i,J_{m})$ et $w(D_k)=1$ sauf pour $D_k=G_i$, auquel cas $w(G_i)=3$. On en d\'eduit : 
$$\frac{(7/3)(p+1)}{w(A_E)w(J_{m})\epsilon(J_i,J_{m})}+O_q(\sqrt{p})\le N(J_{m},J_i)(\nu(J_{m})-\nu(J_i)).$$
En utilisant la majoration donn\'ee par le sous-lemme \ref{JMin1}, on obtient : 
$$\frac{(7/3)(p+1)}{w(A_E)w(J_{m})\epsilon(J_i,J_{m})}+O_q(\sqrt{p})\le \frac{4(p+1)}{(2w(A_E)-1)w(J_{m})\epsilon(J_i,J_{m})} + O_q(\sqrt{p}).$$
Ce qui est contradictoire pour $p\gg q$ comme $w(A_E)>7/2$ d'apr\`es la remarque \ref{genre}.

\end{proof}
\begin{sublem}
\label{egalite}
Sous l'hypoth\`ese que $g(X_0(q))\ge 6$, pour $p\gg q$ les potentiels $\nu(j_{1,1})$, $\nu(j_{1,2}),\dots,\nu(j_{1,l})$ et $\nu(j_{2,1}),\ \nu(j_{2,2}), \dots,  \nu(j_{2,l})$ sont tous \'egaux. En particulier $\nu(J_M)=\nu(J_m)$. 
\end{sublem}
\begin{proof}
Notons $n$ le cardinal de $S_1'$ (et $S_2'$). On rappelle que $S_1'=S_1/w_q$ (cf. notations \ref{notationpar3}), que le cardinal de $S_1$ est $g(X_0(q))+1$, et que $w_q$ laisse fixe au moins les deux sommets $J_1$ et $G_1$ de $S_2.$ Donc $n\ge 2+(g(X_0(q))-1)/2$, on en d\'eduit $n\ge 5$. 

Pour $i=1,2$, on fixe $J_{i,M}$ un sommet ayant un potentiel maximal dans l'ensemble $S_i'\setminus\{G_i\}$. Ainsi que $J_{i,m}$ un sommet dans l'ensemble $S_i'\setminus\{J_i\}$ ayant un potentiel minimal. Supposons que $\nu(J_{1,M})\neq \nu(J_{2,M})$, alors $J_M$ a un potentiel strictement sup\'erieur au potentiel de $n-1$ sommets ou plus de la r\'epartition oppos\'ee. Pour $p\gg q$ cela contredit le sous-lemme \ref{JMcontr}. De m\^eme si $\nu(J_{1,m})\neq \nu(J_{2,m})$, alors $J_m$ a un potentiel strictement inf\'erieur au potentiel de $n-1$ sommets ou plus de la r\'epartition oppos\'ee ce qui contredit le sous-lemme \ref{Jmcontr} pour $p\gg q$. On a donc $\nu(J_{1,M})=\nu(J_{2,M})$ et $\nu(J_{1,m})=\nu(J_{2,m})$ pour $p\gg q$. 

Supposons que $\nu(J_{1,m})= \nu(J_{2,m})<\nu(J_{1,M})= \nu(J_{2,M})$. Si $\nu(J_{1,M})$ est  strictement sup\'erieur au potentiel de 3 sommets ou plus de $S_2'$ on peut appliquer le sous-lemme \ref{JMcontr} pour $J_M=J_{1,M}$. Sinon il y a au plus 2 sommets $D_1,D_2$ de $S_2'$ de potentiel strictement inf\'erieur \`a $\nu(J_{1,M})$, tous les sommets de $S_2'\setminus \{D_1,D_2\} $ sont de potentiel sup\'erieur ou \'egal \`a $\nu(J_{1,M})$, et strictement sup\'erieur \`a $J_{1,m}$. Il y a donc $n-2$ sommets ou plus de $S_2'$ de potentiel strictement sup\'erieur au potentiel de $J_m=J_{1,m}$. Ceci contredit le sous-lemme \ref{Jmcontr}. 
Finalement pour $p\gg q$ on doit avoir $\nu(J_{1,m})= \nu(J_{2,m})=\nu(J_{1,M})= \nu(J_{2,M})$, ce qui d\'emontre la proposition.  
\end{proof}
Dans toute la suite nous supposons que les hypoth\`eses du sous-lemme \ref{egalite} sont satisfaites. Les potentiels $\nu(j_{1,1}), \nu(j_{1,2}),\dots, \nu(j_{1,l})$ et $\nu(j_{2,1}), \nu(j_{2,2}), \dots, \nu(j_{2,l})$ sont donc tous \'egaux \`a $\nu(J_m)=\nu(J_M)$.

\begin{sublem}
\label{congruence}
On a $\nu(J_1)\equiv \nu(J_2) \mod 2$. 
En outre $\nu(J_1)\neq\nu(J_2)$. 
\end{sublem}
\begin{proof}
En appliquant la loi K au sommet $\mathcal{J}$ on obtient : 
$$ \nu(J_1)-\nu(\mathcal{J})+\nu(J_2)-\nu(\mathcal{J})=p+1. $$
En consid\'erant cette \'egalit\'e modulo $2$ on trouve : 
$$ \nu(J_1)\equiv \nu(J_2)  \mod 2, $$
ce qui d\'emontre le premier point. 

Supposons que $\nu(J_1)=\nu(J_2)$. En appliquant la loi K au sommet $J_1$ on trouve : 
$$\sum_{k=1}^lN(J_1,j_{2,k})(\nu(j_{2,k})-\nu(J_1))+N(J_1,G_2)(\nu(G_2)-\nu(J_1))+(\nu(\mathcal{J})-\nu(J_1))=0.$$
En appliquant la loi K au sommet $J_2$ on trouve : 
$$\sum_{k=1}^lN(J_2,j_{1,k})(\nu(j_{1,k})-\nu(J_2))+N(J_2,G_1)(\nu(G_1)-\nu(J_2))+(\nu(\mathcal{J})-\nu(J_2))=0.$$
Or $N(J_1,j_{2,k})=N(J_2,j_{1,k})$, et $N(J_2,G_1)=N(J_1,G_2)$. En outre pour $p\gg q$, par le sous-lemme \ref{egalite}, on a $\nu(j_{2,k})=\nu(j_{1,k})$ et, par le corollaire \ref{Nnonnul}, $N(J_1,G_2)>0$. On en d\'eduit : $\nu(G_1)=\nu(G_2)$. D'autre part en appliquant successivement la loi K aux sommets $\mathcal{J}_1$ et $\mathcal{J}_2$, on trouve (voir preuve du sous-lemme \ref{decr} formule (\ref{resistancesenseries}) ou appliquer la loi des r\'esistances en s\'erie) : 
$$\nu(\mathcal{J}_1)-\nu(G_1)+\frac{1}{2}(\nu(\mathcal{J}_1)-\nu(G_2))=p+1 $$
En consid\'erant cette \'egalit\'e modulo $3$, comme par hypoth\`ese $p\equiv 1 \mod 3$, on a : 
$$\nu(G_2)-\nu(G_1)\equiv 2 \mod 3,$$
En particulier $\nu(G_1)\neq (G_2)$ ce qui est contradictoire. On a donc $\nu(J_1)\neq \nu(J_2)$.
\end{proof}
Le sous-lemme \ref{inferiorite} suivant nous permet de conclure \`a une contradiction avec le sous-lemme \ref{congruence} pour $p\gg q$, ce qui ach\`eve la preuve du lemme \ref{J-J1}. 
\end{proof}
\begin{sublem}
\label{inferiorite}
On a $|\nu(J_1)-\nu(J_2)|<2$. 
\end{sublem}

\begin{proof}
Soient $i$ et $j$ tels que $\{i,j\}=\{1,2\}$ et $\nu(J_i)\le\nu(J_j)$. Supposons que $\nu(J_j)-\nu(J_i)\ge 2$. D'apr\`es le sous-lemme \ref{decr}, $\nu(J_M)>\nu(J_j)$, donc $\nu(J_M)\ge\nu(J_i)+3.$
En outre d'apr\`es les sous-lemmes \ref{decr} et \ref{egalite}, on a $\nu(G_j)>\nu(J_m)=\nu(J_M)\ge \nu(J_i)+3$. Donc pour tout sommet $D\in S_j$, on a $\nu(D)-\nu(J_i)\ge 3$. En appliquant la loi K \`a $J_i$ on trouve : 
$$\sum_{\topbot{D\mapsto J_i}{D\neq \mathcal{J}}}(\nu(D)-\nu(J_i))=\nu(J_i)-\nu(\mathcal{J}),$$
Soit : 
$$ \sum_{\topbot{D\mapsto J_i}{D\neq \mathcal{J}}} 3 \le \nu(J_i)-\nu(\mathcal{J}),$$
c'est-\`a-dire : 
$$3(N(J_i)-1) \le \nu(J_i)-\nu(\mathcal{J}),$$
soit d'apr\`es le lemme \ref{repartition},
$$\frac{3(p-1)}{4} \le \nu(J_i)-\nu(\mathcal{J}).$$
D'autre part en appliquant le loi K \`a $\mathcal{J}$
 $$\nu(J_1)-\nu(\mathcal{J})+\nu(J_2)-\nu(\mathcal{J})=p+1$$
et en utilisant l'in\'egalit\'e : $\nu(J_i)\le\nu(J_j)$, on trouve
$$\nu(J_i)-\nu(\mathcal{J})\le \frac{p+1}{2}.$$
On en d\'eduit : $3(p-1)/4\le (p+1)/2$ ce qui est impossible pour $p>5$ (on a $p>5$ car les hypoth\`eses faites sur $p$ impliquent $p\equiv 1 \mod 12$). 
\end{proof}

\section{Points rationnels sur les courbes de Shimura}

Le crit\`ere de \cite{ParentYafaev} repose sur une description du groupe des caract\`eres du quotient d'enroulement de la composante neutre de la fibre en $p$ du mod\`ele de N\'eron sur $\Z_p$ de la jacobienne $\Jac(X^{pq})^0_{\F_p}$ de la courbe de Shimura $X^{pq}$. D'apr\`es un th\'eor\`eme de Raynaud (cf \cite[12.4]{SGA7-I}), le groupe des caract\`eres de $\Jac(X^{pq})^0_{\F_p}$ est isomorphe au groupe des cycles sur le graphe dual de la fibre en $p$ de la courbe de Shimura $\mathscr{G}(X^{pq}_{\F_p})$. \`A l'aide d'une g\'en\'eralisation de la formule de Gross, Parent et Yafaev prouvent que le groupe des caract\`ere du quotient d'enroulement de $\Jac(X^{pq})^0_{\F_p}$ est isomorphe au groupe des cycles sur $\mathscr{G}(X^{pq}_{\F_p})$ constitu\'es de projections orthogonales pour l'accouplement de monodromie de vecteurs de Gross sur l'espace des diviseurs de degr\'e $0$ sur les ar\^etes du graphe $\mathscr{G}(X^{pq}_{\F_p})$. 

Rappelons que $\mathcal{L}$ et $Y$ (introduits dans le paragraphe 2) d\'esignent respectivement les $\Z$-modules des chemins et des cycles sur le graphe $\mathscr{G}(X^{pq}_{\F_p})$. On consid\`ere \'egalement les $\Q$-espaces vectoriels $\mathcal{L}_{\Q}=\mathcal{L}\otimes \Q$ et $Y_{\Q}=Y\otimes \Q$. 

On note $I_{pq,e}$ l'id\'eal d'enroulement de $\T_{\Gamma_0(pq)}$, c'est-\`a-dire l'ensemble des op\'erateurs de $\T_{\Gamma_0(pq)}$ annulant toutes les formes modulaires primitives $f$ de poids $2$ pour $\Gamma_0(pq)$ telles que $L(f,1)\neq 0$. On a la proposition suivante qui g\'en\'eralise la proposition \ref{GrossApplication}.
\begin{prop}
\label{projgross}
On note par $\mathbb{E}$ l'espace engendr\'e par les projections orthogonales des vecteurs de Gross sur $\mathcal{L}_{\Q}$. On a l'\'egalit\'e : $\mathbb{E}=\mathcal{L}_{\Q}[I_{pq,e}]$.
\end{prop}
Consulter \cite[prop 4.2]{ParentYafaev}. 
En combinant la proposition \ref{projgross} avec le th\'eor\`eme de Raynaud d\'ecrivant le groupe des caract\`eres de $\Jac(X^{pq})^0_{\F_p}$ comme groupe des cycles sur $\mathscr{G}(X^{pq}_{\F_p})$, on obtient la proposition suivante : 
\begin{prop}
\label{caractere}
L'espace $\mathbb{E}\cap Y_{\Q}$ est le groupe des caract\`eres du quotient d'enroulement de $\Jac(X^{pq})_{\F_p}^0$. 
\end{prop}
Consulter \cite[prop 4.3]{ParentYafaev}. 

\begin{thm}
\label{parentyafaev}
Soient $p$ et $q$ deux nombres premiers tels que $q>245$, $q\equiv 3 \mod 4$, $p\equiv 5 \mod 12$, $\legendre{q}{p}=-1$, et $p\gg q$. S'il existe $C\in Y\cap \mathbb{E}$, un chemin ferm\'e sur le graphe $\mathscr{G}(X^{pq}_{\F_p})$ constitu\'e de projections orthogonales de vecteurs de Gross sur $\mathcal{L}^0$, qui contient les deux ar\^etes exceptionnelles de longueur $2$ avec une multiplicit\'e premi\`ere \`a $p$, alors la courbe quotient $X^{pq}/w_q$ est sans point rationnel non sp\'ecial. 
\end{thm} 

\begin{proof}
Consulter \cite[Theorem 5.3]{ParentYafaev}. L'\'enonc\'e que nous donnons est l\'eg\`erement diff\'erent : nous rempla\c cons l'hypoth\`ese \og $C$ est un chemin ferm\'e fait de vecteurs de Gross\fg \  par \og $C$ est un chemin ferm\'e constitu\'e de projection orthogonales de vecteurs de Gross \fg. En consultant la preuve de Parent et Yafaev, on voit que $C$ correspond \`a un caract\`ere du quotient d'enroulement de $\Jac(X^{pq})_{\F_p}^0$, ce qui est \'equivalent, d'apr\`es la proposition \ref{caractere}, au fait que $C$ est dans $Y\cap \mathbb{E}$. 

On peut aussi noter que le vecteur d'Eisenstein appartient \`a l'espace engendr\'e par les vecteurs de Gross, et donc que la projection orthogonale sur $\mathcal{L}^0$ d'un chemin en vecteurs de Gross est \'egalement constitu\'e de vecteurs de Gross. 

\end{proof}
Nous allons utiliser le travail du troisi\`eme paragraphe pour g\'en\'eraliser ce crit\`ere en supprimant la condition $p\equiv -1\mod 3$. Ainsi les hypoth\`eses de congruences sur $p$ et $q$ sont \'equivalentes aux conditions du cas non ramifi\'e de Ogg ((2) th\'eor\`eme \ref{casdeogg}). 
\begin{thm}
\label{parentyafaevgeneralise}
Soient $p$ et $q$ deux nombres premiers tel que $q>245$, $q\equiv 3 \mod 4$, $p\equiv 1 \mod 4$, $\legendre{q}{p}=-1$, et $p\gg q$. S'il existe $C\in Y\cap \mathbb{E}$, un chemin ferm\'e sur le graphe $\mathscr{G}(X^{pq}_{\F_p})$ constitu\'e de projections orthogonales de vecteurs de Gross sur $\mathcal{L}^0$, qui contient les deux ar\^etes exceptionnelles de longueur $2$ avec une multiplicit\'e premi\`ere \`a $p$, alors la courbe quotient $X^{pq}/w_q$ est sans point rationnel non sp\'ecial. 
\end{thm}
\begin{proof}
Dans la preuve de \cite[Theorem 5.3]{ParentYafaev} l'hypoth\`ese $p\equiv -1\mod 3$ n'est utilis\'ee que dans la preuve du lemme \cite[lemma 3.1.3]{ParentYafaev}. Or nous avons g\'en\'eralis\'e ce lemme (cf. lemme \ref{general}) en supprimant cette condition. 
\end{proof}

Pour la preuve du th\'eor\`eme \ref{resfinal}, nous utiliserons la pr\'esente proposition d\'ecrivant l'image du vecteur d'Eisenstein-Shimura $a_E$ par les applications $s_*$ et $t_*$. 
\begin{prop}
\label{repar}
Si $A_E$ et $a_E$ d\'esignent respectivement les vecteurs d'Eisenstein sur le graphe modulaire $\mathscr{G}(X_0(q)_{\F_q})$ et le graphe de Shimura $\mathscr{G}(X^{pq}_{\F_p})$ (cf. d\'efinitions \ref{EisensteinModulaire} et \ref{EisensteinShimura}), on a : 
$s_*(a_E)=(p+1)A_E$, et $t_*(a_E)=(p+1)A_E$. 
\end{prop}

\begin{proof}
Par d\'efinition : 
$$a_E=\sum_{i=1}^n \frac{1}{\card(\End(e_i)^*/\{\pm 1\})} e_i=\sum_{\topbot{(E,C_p)}{\text{\`a isom. pr\`es}}}\frac{1}{(\End(E,C_p)^*/\{\pm 1\})} (E,C_p).$$
O\`u $(E,C_p)$ parcourt l'ensemble des couples constitu\'e d'une courbe elliptique supersinguli\`ere $E$ en caract\'eristique $q$ et d'un sous groupe d'ordre $p$ de $E$
$$s_*(a_E)=\sum_{k=1}^{g+1} \bigg( \sum_{\topbot{(E_{j_k},C_p)}{\text{\`a isom. pr\`es}}} \frac{1}{\card(\End(E_{j_k},C_p)^*/\{\pm 1\})}  \bigg) E_{j_k}. $$
Pour $k$ fix\'e, il y a $p+1$ couples de la forme $(E_{j_k},C_p)$. Par la formule des classes pour le groupe $\End(E_{j_k})^*/\{\pm 1\}$ op\'erant sur l'ensemble des $(E_{j_k},C_p)$, on a : 
$$ \sum_{\topbot{(E_{j_k},C_p)}{\text{\`a isom. pr\`es}}} \frac{1}{\card(\End(E_{j_k},C_p)^*/\{\pm 1\})}=\frac{p+1}{\card(\End(E_{j_k})^*/\{\pm 1\})}.$$
Donc $s_*(a_E)=(p+1) A_E$. On montre de m\^eme que $t_*(a_E)=(p+1) A_E.$

\end{proof}

Prouvons maintenant notre r\'esultat principal : le th\'eor\`eme \ref{resfinal} dont l'\'enonc\'e est rappel\'e ci-dessous. 
\begin{thm}
Soit $q>245$ un nombre premier avec $q \equiv 3 \mod 4$. Il existe une borne $B_q$ d\'ependant de $q$ telle que si $p$ est un nombre premier v\'erifiant $p\equiv 1 \mod 4$, $\legendre{p}{q}=-1$ et $p\ge B_q$, alors la courbe $X^{pq}/w_q$ est sans point rationnel non sp\'ecial. 
\end{thm}

\begin{proof}
Soit $l\neq p,q$ un nombre premier. D'apr\`es la proposition \ref{Eisenstein}, le vecteur d'Eisenstein modulaire $A_E$ sur $\mathscr{G}((X_0(q))_{\F_q})$ appartient au $\Q-$espace engendr\'e par les vecteurs de Gross de la forme $\Gamma_D$, pour $D=-4l^{2n}$, o\`u $n\ge 1$. Ces vecteurs $\Gamma_D$ correspondent \`a des sous-ordres stricts $O_D$ de l'ordre $O_{-4}$ de conducteurs $l^n$. On peut donc \'ecrire $A_E$ comme combinaison lin\'eaire de ces vecteurs : 
$$\lambda_0A_E =\sum_{n=1}^{N} \lambda_n \Gamma_{-4l^{2n}},$$
avec $\lambda_n\in\Z$ et $\lambda_0\neq 0$. Pour ne pas avoir de d\'enominateurs on choisit tous les $\lambda_n$ multiples de $12$. 
D'apr\`es la proposition \ref{intersection}, il existe $B$ une borne telle que, pour $p\ge B$, chacun des vecteurs $\gamma_{-4l^{2n}}$ soit sans ar\^ete commune avec les vecteurs $\gamma_{-4}$ et $\gamma_{-3}$. On suppose dans la suite $p\ge B$, de telle sorte que les vecteurs $\Gamma_{-4l^{2n}}$ ne contiennent pas d'ar\^ete exceptionnelle de longueur $2$ ou $3$. On impose en outre que $p$ ne divise pas $\lambda_0$. On suppose \'egalement que $p \gg q$ de telle sorte qu'on puisse appliquer le th\'eor\`eme \ref{parentyafaevgeneralise}. 

Consid\'erons le chemin en vecteurs de Gross $C=\sum_{n=1}^{N} \lambda_n \gamma_{-4l^{2n}}$ sur le graphe dual de la fibre en $p$ de la courbe de Shimura $X^{pq}$. D'apr\`es l'hypoth\`ese, $p\equiv 1 \mod 4$, $p$ est scind\'e dans $O_{-4}$ et dans tous ses sous-ordres stricts $O_{-4l^{2n}}$. En outre chacun de ces vecteurs ne contient que des ar\^etes de longueur $1$. On peut appliquer le corollaire \ref{vecteurdegrossshimura} : 
$$ s_*(C)=\sum_{n=1}^{N} \lambda_n s_*(\gamma_{-4l^{2n}})=\sum_{n=1}^{N} \lambda_n 4\Gamma_{-4l^{2n}}=4\lambda_0A_E.$$
De m\^eme $t_*(C)=4\lambda_0(A_E)$. D'autre part d'apr\`es la proposition \ref{repar}, on a : $s_*(a_E)=t_*(a_E)=(p+1)A_E$. 
Le chemin $C_0=(p+1)C-4\lambda_0 a_E$, est un chemin ferm\'e. Pour cela il suffit de voir que $s_*(C_0)=t_*(C_0)=0$. 

D'autre part, le chemin $C_0$ est la projection orthogonale pour l'accouplement de monodromie du chemin en vecteurs de Gross $(p+1)C$ sur l'espace $\mathcal{L}^{0}$ des diviseurs de degr\'e $0$. Donc $C_0$ appartient \`a $\mathbb{E}$. Finalement $C_0$ est un chemin de $\mathbb{E}\cap Y$. D'apr\`es les hypoth\`eses faites sur $p$, les vecteurs du chemin $C$ ne contiennent pas d'ar\^ete exceptionnelle de longueur $2$. D'autre part le vecteur d'Eisenstein $a_E$ contient chacune des deux ar\^etes exceptionnelle de longueur $2$ avec multiplicit\'e $1/2$. Le chemin $C_0$ passe par chacune de ces ar\^etes exceptionnelles avec multiplicit\'e $2\lambda_0$. Donc on peut appliquer le th\'eor\`eme de Parent-Yafaev \`a ce chemin. On en d\'eduit que la courbe de Shimura $X^{pq}/w_q$ est sans point rationnel non sp\'ecial.

\end{proof}

\end{document}